\documentclass[11pt,draft]{article}

\usepackage{pdfsync}

\usepackage{a4wide}
\usepackage{amsmath,amsthm}
\usepackage{amsfonts}
\usepackage{color}
\usepackage{hyperref}
\usepackage{enumerate}

\newcommand{\mm}{\mathbb{M}^n_k(c)}
\newcommand{\VV}{\ep\,I\times_a \mathbb{E}^{n+1}}
\newcommand{\mmm}{\ep\,I\times_a\mm}
\newcommand{\R}{\mathbb{R}}
\newcommand{\nab}{\bar\nabla}
\newcommand{\nat}{\widetilde{\nabla}}
\newcommand{\na}{\nabla}
\newcommand{\ep}{\varepsilon}
\newcommand{\ti}[1]{\tilde{#1}}
\newcommand{\dt}{\partial_t}
\newcommand{\grad}{\mathrm{grad}}
\newcommand{\X}{\mathbf{X}}

\newtheorem{lemma}{Lemma}
\newtheorem{proposition}{Proposition}
\newtheorem{theorem}{Theorem}
\newtheorem{corollary}{Corollary}
\newtheorem{definition}{Definition}
\newtheorem{examp}{Example}
\newenvironment{example}[1]{\begin{examp} #1 \newline \rm }{\end{examp}}

\newcommand{\produ}[1]{\langle #1\rangle}
\newcommand{\E}{\overline{E}}
\newcommand{\PP}{\bar{P}^{n+1}}
\newcommand{\DD}{\mathbf{D}}
\renewcommand{\S}{\mathbf{S}}

\title{A Fundamental Theorem for Hypersurfaces in Semi-Riemannian Warped
Products}
\author{Marie-Am\'elie Lawn and Miguel Ortega}
\date{\today}

\begin{document}

\maketitle

\begin{abstract}
We give necessary and sufficient conditions for a semi-Riemannian manifold of arbitrary signature to be locally isometrically immersed into a warped product $\pm I\times_a\mathbb{M}^n(c)$, where $I\subset\mathbb{R}$ and $\mathbb{M}^n(c)$ is a semi-Riemannian space of constant nonzero sectional curvature. Then, we describe a way to use the structure equations of such immersions to construct foliations of marginally trapped surfaces in a four-dimensional Lorentzian spacetimes. We point out that, sometimes, Gau\ss\, and Codazzi equations are not sufficient to ensure the existence of a local isometric immersion of a semi-Riemannian manifold as a hypersurface of another manifold. We finally give two low-dimensional examples to illustrate our results.
\end{abstract}

\section{Introduction}

One of the fundamental problems in submanifold theory deals with the existence of isometric immersions from one manifold into another. The Gau\ss, Ricci and Codazzi equations are very well-known as \textit{the structure equations}, meaning that any submanifold of any semi-Riemannian manifold must satisfy them. A classical result states that, conversely, they are necessary and sufficient conditions for a Riemannian $n$-manifold to admit a (local) immersion in the Euclidean $(n+1)$-space. In addition,  E. Cartan developed the so-called \textit{moving frames} technique, obtaining a necessary and sufficient condition to construct a map from a (differential) manifold $M$ into a Lie group. If a Lie group $\mathbf{G}$ is a group of diffeomorphisms of a manifold $P$, Cartan's technique then may provide a map from $M$ to $P$ with nice properties. Sometimes, the map from $M$ to $\mathbf{G}$ can exist thanks to Gau\ss,  Codazzi and Ricci equations, like for instance in \cite{KS}.

Another point of view is the celebrated Nash Theorem, which states that any Riemann manifold can be embedded in the Euclidean space, but at the price of a high codimension. Following that line, O. M\"uller and M. S\'anchez obtained a characterization of the Lorenztian manifolds which can be embedded in a high dimensional Minkowski space (see \cite{MS}.)

On the other hand, B.~Daniel obtained in \cite{D} a \textit{fundamental theorem} for hypersurfaces in the Riemannian products $\mathbb{S}^n\times\R$ and $\mathbb{H}^n\times\R$,  looking for tools to work with minimal surfaces in such manifolds when $n=2$. J.~Roth generalized B. Daniel's theorem to spacelike hypersurfaces in some Lorentzian products (see \cite{JR}). In their works, they needed some extra tools such as a tangent vector field $T$ to the submanifold and some functions, in order to obtain the local metric immersions into the desired ambient spaces. Note that the Ricci equation provides no information for hypersurfaces. 

Our main aim is to obtain a \textit{fundamental theorem} for non-degenerate hypersurfaces in a semi-Riemannian warped product, namely  $\mmm$, where $\ep=\pm 1$, $a:I\subset\R\rightarrow\R^+$ is the scale factor and $\mathbb{M}_k^n(c)$ is the semi-Riemannian space form of index $k$ and constant curvature $c=\pm 1$. For a hypersurface $M$ in $\mmm$, the vector field $\dt$ ($t\in I$) decomposes in its tangent and normal parts, i.~e., $\dt = T+\ep_{n+1}T_{n+1}e_{n+1}$ where $e_{n+1}$ is a (local) normal unit vector field, $\ep_{n+1}=\pm 1$ shows its causal character and $T_{n+1}$ is the corresponding coordinate.  In addition to the shape operator $A$, on Gau\ss\, and Codazzi equations there appear the vector field $T$, its dual 1-form $\eta$, some constants as well as some functions like $T_{n+1}$.  However, the covariant derivative of $T$ must satisfy a specific formula, which cannot be obtained from Gau\ss\, and Codazzi equations by the authors. Based  on these  necessary conditions, we state in Definition \ref{structureconditions} all needed tools on an abstract semi-Riemannian manifold $M$, for the existence of a (local) metric immersion $\chi: \mathcal{U}\subset M \rightarrow \mmm$ (see Theorem \ref{main}.) Later, we apply this result to non-degenerate hypersurfaces of a Friedman-Lem\^etre-Rober\-t\-son-Walker 4-spacetimes (RW 4-spacetimes.) In Corollary \ref{RW}, we show  sufficient conditions for such hypersurfaces to exist. 

We would like to point out that our computations, as well as B. Daniel and J. Roth's results, show that  Gau\ss\, and Codazzi equations are not sufficient to ensure the existence of a local isometric immersion of a given Riemannian manifold endowed with a second fundamental form in a spacetime as a spacelike hypersurface. 

Next, if we admit in a very wide sense that a \textit{horizon} in a 4-spacetime is a 3-dimensional hypersurface which is foliated by marginally trapped surfaces (i.~e., surfaces whose mean curvature vector is timelike), then we describe a condition to obtain non-degenerate horizons in RW 4-spacetimes in our framework (see Corollary \ref{horizon}).  

We end the paper with two low-dimensional examples to illustrate the theoretical results. The first one  describes a surface in a RW toy model $\mathbb{S}^2\rightarrow -I\times_a\mathbb{S}^2$,  a (simple) graph over a rest space $\{t_0\}\times\mathbb{S}^2$. The second example is a  helicoidal surface in $ -I\times_a\mathbb{H}^2$.

\section{Preliminaries}
Let $(P,g_P)$ be a semi-Riemannian manifold of dimension $\dim P=m$. We consider
a smooth function $a:I\subset\R\rightarrow \R^{+}$, a (sign) constant $\ep=\pm
1$ and the warped product
\[ \bar{P}^{m+1} \ = \ \varepsilon I\times_a P, \quad \langle,\rangle =
\ep\mathrm{d}t^2+a^2(t)g_P.
\]
Clearly, the unit vector field $\frac{\partial}{\partial t}= \partial_t$ will play a crucial role on the manifold $\PP$. 
We will use the following convention for the curvature operator $\mathcal{R}$ of a connection $\mathcal{D}$:
\[\mathcal{R}(X,Y)Z= \mathcal{D}_X\mathcal{D}_YZ-\mathcal{D}_Y\mathcal{D}_XZ-\mathcal{D}_{[X,Y]}Z.
\] 
Let $\bar{R}_P$ and $R_P$ be the curvature operator of $\PP$ and $P$, respectively. Let $\mathbf{D}$ be the
Levi-Civita connection of $\PP$.  We recall the following formulae from \cite{ON}.  
\begin{lemma} \label{lemma-on-1} On the semi-Riemannian manifold $\PP$, the
following statements hold, for any $V,W$ lifts of vector fields tangent to $P$:
\begin{enumerate}
\item \label{uno} $\mathbf{D}_{\partial_t}\partial_t=0$,
$\mathbf{D}_V\partial_t=\frac{a'}{a}\,V$,
\item $\mathrm{grad}(a)=\ep a'\partial_t$.
\item \label{nablas} $\mathbf{D}_VW = \nabla_V^PW-\frac{\ep a'}{a}\produ{V,W}\dt.$
\item \label{cuatro} $\bar{R}_P(V,\dt)\dt = -\frac{a''}{a}V$,
\label{cinco} $\bar{R}_P(\partial_t,V)W=-\ep \frac{a''}{a}\langle V,W\rangle\partial_t$, 
\label{seis} $\bar{R}_P(V,W)\dt=0$.
\end{enumerate}
\end{lemma}
\begin{proof} Note that the definition of the curvature operator on
\cite{ON} has the opposite sign than the usual one. We show a proof of item
(\ref{cuatro}). By recalling $\DD_{\dt}\dt =0$ and $[V,\dt]=0$, we have
$\bar{R}_P(V,\dt)\dt=\DD_{V}\DD_{\dt}\dt-\DD_{\dt}\DD_{\dt}V-\DD_{[V,\dt]}\dt
=-\DD_{\dt}\DD_{\dt}V =-\DD_{\dt}\left(\frac{a'}{a}V\right)
=-\frac{a''a -(a')^2}{a^2}V-\frac{a'}{a}\DD_{\dt}V =
-\frac{a''}{a}V.$ Next, we show 
 $\bar{R}_P(\dt,V)W = 
 -\frac{\produ{V,W}}{a}\DD_{\dt}(\mathrm{grad} a)= -\frac{\produ{V, W}}{a}\DD_{\dt} (\ep a'\dt)
= -\ep \frac{\produ{V, W}a''}{a}\dt.$ Finally, $\bar{R}_P(V,W)\dt=0$ is a direct consequence of item \ref{uno}.
\end{proof}

Now, let $\mathcal{M}$ be a non-degenerate hypersurface of $\bar{P}^{m+1}$, with $\nabla^{\mathcal{M}}$ its Levi-Civita connection, $\sigma$ the second fundamental form and $R_{\mathcal{M}}$ the curvature operator of $\mathcal{M}$, respectively. Given a (local) unit  normal vector field $\nu$ of $\mathcal{M}$ in $\bar{P}^{m+1}$, with $\delta=\produ{\nu,\nu}=\pm 1$, let $\mathcal{A}$ be the shape operator associated with $\nu$. The Gau\ss\, and Weingarten's formulae are
\[ \mathbf{D}_X Y = \nabla^{\mathcal{M}}_XY +\sigma(X,Y), \quad 
\quad \mathbf{D}_X\nu = -\mathcal{A}X,
\] 
for any $X,Y\in T\mathcal{M}.$ The second fundamental form can be written as 
\[ \sigma (X,Y) \ = \ \delta \produ{\mathcal{A}X,Y}\nu, \ \mbox{for any } X,Y\in T\mathcal{M}.
\]
Recall that the \textit{mean curvature vector} of $\mathcal{M}$ is defined by
\[ \vec{H}=\frac{1}{\dim(\mathcal{M})} \mathrm{Tr}(\sigma).
\]

Next, the Codazzi equation of $\mathcal{M}$ takes the general form
$(\bar{R}_P(X,Y)Z)^{\perp} = (\mathbf{D}_X\sigma)(Y,Z)-(\mathbf{D}_Y\sigma)(X,Z)$, 
for any $X,Y,Z$ tangent to $\mathcal{M}$, which is equivalent to
\begin{equation}
\label{eq-codazzigeneral}
\bar{R}_P(X,Y,Z,\nu) = 
\produ{(\mathbf{D}_X\mathcal{A}) Y - (\mathbf{D}_Y\mathcal{A}) X,Z},
\end{equation}
for any $X,Y,Z\in T\mathcal{M}$. Further, the general Gau\ss~  equation is given by
\begin{align}\label{eq-Gauss}
\bar{R}_P(X,Y,Z,W)&=R_{\mathcal{M}}(X,Y,Z,W)
-\langle\sigma(Y,Z),\sigma(X,W)\rangle+\langle\sigma(Y,W),\sigma(X,Z)\rangle \\
&=R_{\mathcal{M}}(X,Y,Z,W)
-\delta \produ{\mathcal{A}Y,Z}\produ{\mathcal{A}X,W}+\delta \produ{\mathcal{A}Y,W}\produ{\mathcal{A}X,Z}, \nonumber
\end{align}
with $X,Y,Z,W$ tangent to $\mathcal{M}$.\\

We consider now the special case where the manifold $P=\mathbb{E}^{n+1}=\R^{n+1}_k$, i.~e., the standard Euclidean semi-Riemannian space of dimension $n+1\geq 3$ and index $k$. Following the previous notation, we construct $\widetilde{P}^{n+2}=\VV.$ Let $\widetilde{R}$ be the curvature tensor of $\widetilde{P}^{n+2}$. We have
\begin{proposition}\label{tildeR} Let $X,Y,Z,W\in\Gamma(T\widetilde{P}^{n+2})$. \begin{gather*} 
\widetilde{R}(X,Y,Z,W)=
\ep\frac{(a')^2}{a^2} \Big( \langle X,Z\rangle\langle Y,W\rangle -\langle Y,Z\rangle\langle X,W\rangle \Big)  
\\ 
+\left(\frac{a''}{a} -\frac{(a')^2}{a^2}\right)\Big(
\langle X,Z\rangle\langle Y,\dt \rangle\langle W,\dt \rangle
-\langle Y,Z\rangle\langle X,\dt \rangle\langle W,\dt \rangle
\\
- \langle X,W\rangle\langle Y,\dt \rangle\langle Z,\dt \rangle
+ \langle Y,W\rangle\langle X,\dt \rangle\langle Z,\dt \rangle
\Big).
\end{gather*}
\end{proposition}
\begin{proof} Let $X=\tilde{X}+x\dt=\tilde{X}+\ep\produ{X,\dt}\dt$, where $\tilde{X}$ is a vector field tangent to $\widetilde{P}^{n+2}$. We will use similar notations for other vector fields.  In particular, we see that $\produ{\tilde X,\tilde Y}=\produ{X,Y}-\ep\produ{X,\dt}\produ{Y,\dt}$. By using the symmetry properties of the curvature tensor, we get
\begin{align*}
\tilde{R}(X,Y,Z,W) &=\tilde{R}(\tilde{X},\tilde{Y},\tilde{Z},\tilde{W})+\tilde{R}(\tilde{X},\tilde{Y},
\tilde{Z},w\partial_t)
+\tilde{R}(\tilde{X},\tilde{Y},z\partial_t,\tilde{W})\\
&+\tilde{R}(\tilde{X},y\partial_t,\tilde{Z},\tilde{W})
+\tilde{R}(\tilde{X},y\partial_t,\tilde{Z},w\partial_t)
+\tilde{R}(\tilde{X},y\partial_t,z\partial_t,\tilde{W})&  \\
&+\tilde{R}(x\partial_t,\tilde{Y},\tilde{Z},\tilde{W})
+\tilde{R}(x\partial_t,\tilde{Y},\tilde{Z},w\partial_t)+\tilde{R}(x\partial_t,\tilde
{Y},z\partial_t,\tilde{W}) &
\end{align*}
By Lemma \ref{lemma-on-1}, we obtain directly
\begin{align*}
\tilde{R}(\tilde{X},\tilde{Y},z\partial_t,\tilde{W})&=0,\quad 
\tilde{R}(\tilde{X},\tilde{Y},\tilde{Z},w\partial_t)=-\tilde{R}(\tilde{X},\tilde{Y}
,w\partial_t,\tilde{Z})=0,&\\
\tilde{R}(x\partial_t,\tilde{Y},\tilde{Z},\tilde{W})&=\tilde{R}(\tilde{Z},\tilde{W},
x\partial_t,\tilde{Y})=0,\quad
\tilde{R}(\tilde{X},y\partial_t,\tilde{Z},\tilde{W})=-\tilde{R}(\tilde{Z},\tilde{W}
,y\partial_t,\tilde{X})=0.
\end{align*}
Since the curvature tensor of $\mathbb{E}^{n+1}$ vanishes, by \cite[p. 210]{ON}, we get
\[
\tilde{R}(\tilde{X},\tilde{Y},\tilde{Z},\tilde{W})=-\ep \frac{(a')^2}{a^2}
(\langle\tilde{Y},\tilde{Z}\rangle\langle\tilde{X},\tilde{W}
\rangle-\langle\tilde{X},\tilde{Z}\rangle\langle\tilde{Y},\tilde{W}\rangle).
\]
Moreover, with Lemma \ref{lemma-T}.2, using
again Lemma \ref{lemma-on-1}.4 and as $\langle\partial_t,\partial_t\rangle=\ep$, 
\[
\tilde{R}(\tilde{X},y\partial_t,\tilde{Z},w\partial_t)=
-\tilde{R}(y\partial_t,\tilde{X},\tilde{Z},w\partial_t)
=\ep\frac{a''}{a}\langle\tilde{X},\tilde{Z}\rangle\langle y\partial_t,
w\partial_t\rangle
=\frac{a''}{a}\langle Y,\dt \rangle \langle
W,\dt \rangle\langle\tilde{X},\tilde{Z}\rangle.
\]
By similar computations, we obtain 
\begin{align}\label{equation_curvature_tildeP}
\tilde{R}(X,Y,Z,W)=&-\ep\frac{(a')^2}{a^2} \Big(\langle\tilde{Y},\tilde{Z}\rangle\langle\tilde{X},\tilde{W}
\rangle-\langle\tilde{X},\tilde{Z}\rangle\langle\tilde{Y},\tilde{W} \rangle\Big) \nonumber\\
&+\frac{a''}{a}\Big(
\langle Y,\dt \rangle \langle W,\dt \rangle\langle\tilde{X},\tilde{Z}\rangle
-\langle Y,\dt \rangle\langle Z,\dt \rangle\langle\tilde{X},\tilde{W}\rangle  \\
&-\langle X,\dt \rangle\langle W,\dt \rangle\langle\tilde{Y},\tilde{Z}\rangle\nonumber
+\langle X,\dt \rangle\langle Z,\dt \rangle\langle\tilde{Y},\tilde{W}\rangle
\Big).
\end{align}
Now, straightforward computations yield
\begin{align*}
&\langle\tilde{Y},\tilde{Z}\rangle\langle\tilde{X},\tilde{W}
\rangle-\langle\tilde
{X},\tilde{Z}\rangle\langle\tilde{Y},\tilde{W}\rangle 
+\langle Y,\dt \rangle\langle Z,\dt \rangle\langle X,\dt \rangle\langle W,\dt \rangle
-\langle X,Z\rangle\langle Y,W\rangle
\\
&\quad +\ep\langle X,Z\rangle\langle Y,\dt \rangle\langle W,\dt \rangle
+\ep\langle X,\dt \rangle\langle Z,\dt \rangle\langle Y,W\rangle
-\langle X,\dt \rangle\langle Z,\dt \rangle\langle Y,\dt \rangle\langle W,\dt \rangle
\\
&=
\langle Y,Z\rangle\langle X,W\rangle
-\ep\langle Y,Z\rangle\langle X,\dt \rangle\langle W,\dt \rangle
-\ep\langle Y,\dt \rangle\langle Z,\dt \rangle\langle X,W\rangle
-\langle X,Z\rangle\langle Y,W\rangle
\\
&\quad
+\ep\langle X,Z\rangle\langle Y,\dt \rangle\langle W,\dt \rangle
+\ep\langle X,\dt \rangle\langle Z,\dt \rangle\langle Y,W\rangle,
\end{align*}
and
\begin{align*}
&
\langle Y,\dt \rangle \langle W,\dt \rangle\langle\tilde{X},\tilde{Z}\rangle
-\langle Y,\dt \rangle\langle Z,\dt \rangle\langle\tilde{X},\tilde{W}\rangle
-\langle X,\dt \rangle\langle W,\dt \rangle\langle\tilde{Y},\tilde{Z}\rangle
+\langle X,\dt \rangle\langle Z,\dt \rangle\langle\tilde{Y},\tilde{W}\rangle \\
&=
\langle Y,\dt \rangle \langle W,\dt \rangle\langle X,Z\rangle
-\langle Y,\dt \rangle\langle Z,\dt \rangle\langle X,W\rangle
-\langle X,\dt \rangle\langle W,\dt \rangle\langle Y,Z\rangle
+\langle X,\dt \rangle \langle Z,\dt \rangle\langle Y,W\rangle.
\end{align*}
By inserting in \eqref{equation_curvature_tildeP}, we finally get the result.
\end{proof}
Let $\mm$ be the semi-Riemannian space form of constant sectional
curvature $c=\pm 1$ and index $k$, with metric $g$ and let 
$\PP=\mmm$, with metric $\produ{,}$. We denote by $\bar{R}$
the curvature operator of $\mmm$. Also, we put 
\begin{gather*}
\mathbb{E}^{n+1}=
\begin{cases}
\mathbb{R}^{n+1}_k ,& \textrm{ if }\mathbb{M}^n_k(c)=\mathbb{S}^n_k, \quad
c=+1, \\
\mathbb{R}^{n+1}_{k+1},& \textrm{ if }\mathbb{M}_k^n(c)=\mathbb{H}^n_k,\quad
c=-1,
\end{cases}
\end{gather*}
with its standard metric $g_o$ and Levi-Civita connection $\nabla^o$. We recall that
\[ \mathbb{S}^n_k =  \{ p\in \mathbb{E}^{n+1} : g_o(p,p)=+1\},
\quad
\mathbb{H}^n_k =  \{ p\in \mathbb{E}^{n+1} : g_o(p,p)=-1\}.
\]

From the usual totally umbilical embedding $\Xi:\mm\rightarrow\mathbb{E}^{n+1}$, we construct the following isometric embedding 
\[\widetilde{\Xi}: (\mmm,\produ{,})\longrightarrow (\varepsilon I\times_a\mathbb{E}^{n+1},\produ{,}_2), \quad (t,p)\mapsto (t,\Xi(p)).
\]
In the sequel, for the sake of simplicity, we will also use the notation $\produ{,}_2=\produ{,}$. Let $\nat$ and $\nab$ be the Levi-Civita connection on $\VV$ and $\mmm$, respectively. It is well-known that $\xi=\Xi/c$  is a unit normal vector field satisfying $\nabla^o_X\xi = X/c$ for any $X$ tangent to $T_p\mathbb{M}_k^n(c)$. Thus, we can
consider the normal vector field of $\tilde{\Xi}:\mmm\rightarrow\VV$ as
\[e_{0}(t,p)=(0,\xi (p)/a(t))=\big(0,p/(c\,a(t))\big), \quad
\mathrm{for\ any\ } (t,p)\in\mmm.
\]
We also set $\ep_{0}=\produ{e_{0},e_{0}}=\pm 1$. Since $\mathbb{M}_k^n(c)$ lies naturally in $\mathbb{E}^{n+1}$, the normal vector field $\xi$ satisfies $g_o(\xi,\xi)=c$. In addition, $\ep_0 = \produ{e_0,e_0}=\produ{(0,p/(ac)),(0,p/(ac))} = a^2 g_o(p,p) /(a^2c^2) =c$. In this way, by Lemma \ref{lemma-on-1},
\[ \widetilde{\nabla}_{\dt}e_{0} =\frac{-a'}{a^2}(0,\xi)+\frac{1}{a}\widetilde{\nabla}_{\dt} (0,\xi)
= \frac{-a'}{a^2}(0,\xi) +\frac{1}{a}\frac{a'}{a} (0,\xi) =0.\]
Moreover, if $(0,Z)\perp e_{0}$, then $Z\perp \xi$, so that
\[
\widetilde{\nabla}_{(0,Z)}e_{0} =
\frac{1}{a}\widetilde{\nabla}_{(0,Z)} (0,\xi) = \frac{1}{a} \Big(\nabla_Z^0\xi -\frac{\produ{(0,Z),(0,\xi)}}{a}\grad a\Big)= \frac{1}{ac}(0,Z).
\]
This means that the Weingarten operator $S$ associated with $e_{0}$ has the expression
\begin{equation}
\label{shapeoperatoren+2} S Y = \frac{-1}{ac} ( Y - \ep \produ{Y,\dt}\dt ),
 \mbox{ for any $Y\in T\bar{P}^{n+1}$.}
\end{equation}
\begin{proposition}\label{barR}
The curvature tensor of $\bar{P}^{n+1}=\mmm$ is
\begin{align}
\bar{R}(X,Y,Z,W) &= \left( \ep\frac{(a')^2}{a^2}-\frac{\ep_{0}}{a^2}  \right)\Big(
 \langle X,Z\rangle\langle Y,W\rangle -\langle Y,Z\rangle\langle X,W\rangle \Big)\nonumber \\
& +\left(\frac{a''}{a} -\frac{(a')^2}{a^2}+\frac{\ep\ep_{0}}{a^2}\right)\Big(
\langle X,Z\rangle\langle Y,\dt \rangle\langle W,\dt \rangle
-\langle Y,Z\rangle\langle X,\dt \rangle\langle W,\dt \rangle
\\
&- \langle X,W\rangle\langle Y,\dt \rangle\langle Z,\dt \rangle
+ \langle Y,W\rangle\langle X,\dt \rangle\langle Z,\dt \rangle
\Big), \nonumber
\end{align}
for any $X, Y, Z, W$ in $T\bar{P}^{n+1}$.
\end{proposition}
\begin{proof}
We just need to resort to (\ref{eq-Gauss}), Proposition \ref{tildeR} and (\ref{shapeoperatoren+2}). 
\end{proof}

\section{Hypersurfaces}

Let $M^n$, $n\geq 2$, be an immersed, non-degenerate hypersurface in
$\bar{P}^{n+1}$. Let $\na$ be the Levi-Civita connections on $M$. Let $e_{n+1}$ be a (locally defined) normal unit vector field to $M$, with $\ep_{n+1}=\produ{e_{n+1},e_{n+1}}=\pm 1$. Along  $M$, the vector field $\partial_t$ can be decomposed as its tangent and normal parts, i.~e.,
\[ \partial_t = T+ f\,e_{n+1},\]
where $T$ is  tangent to $M$ and $f=\ep_{n+1} \produ{\dt,e_{n+1}}$. We also define the 1-form on $M$ given by $\eta(X)=\produ{X,T}$, for any $X\in \Gamma(TM)$.  Given a tangent vector $X$ to $M$, we again decompose it in the  part tangent to $\{t\}\times \mm$ and its component in the direction of
$\partial_t$ as $X=\ti{X}+x\partial_t$. Similarly, $Y=\ti{Y}+y\partial_t$ and $e_{n+1}=\ti{e}_{n+1}+n\partial_t$.
\begin{lemma}\label{lemma-T} Under the previous conditions,
\begin{enumerate}
\item $\ep n=\langle e_{n+1},\partial_t\rangle = \ep_{n+1} f$, 
$\ep x=\langle X,\partial_t\rangle =\langle X,T\rangle$,
$\ep y=\langle Y,\partial_t\rangle =\langle Y,T\rangle$,
\item $\langle \ti{X},\ti{e}_{n+1}\rangle = -\ep \ep_{n+1} f\langle X,T\rangle$,
\item $\langle \ti{X},\ti{Y}\rangle =\langle X,Y\rangle -\ep\langle
X,T\rangle\langle Y,T\rangle$.
\end{enumerate}
\end{lemma}
\begin{proof} From ${e}_{n+1} =\tilde{e}_{n+1}+n\partial_t$, it is immediate that $\ep
n=\produ{\partial_t,{e}_{n+1}}=\ep_{n+1}  f$. Further, $\produ{X,\dt}=\produ{X,T+f{e}_{n+1}}=\produ{X,T}=\produ{\ti{X}+x\dt,\dt}=\ep x$.
Next,
$0=\produ{X,{e}_{n+1}}=\produ{\ti{X},\ti{e}_{n+1}}+xn\ep = \produ{\ti{X},\ti{e}_{n+1}}+\ep\ep_{n+1} f\produ{X,T}$.
Finally,
$\produ{X,Y}=\produ{\ti{X}+x\dt,\ti{Y}+y\dt} = \produ{\ti{X},\ti{Y}}+\ep xy
= \produ{\ti{X},\ti{Y}}+\ep \produ{X,T}\produ{Y,T}$.
\end{proof}

Let $A$ be the shape operator of $M$ associated with  $e_{n+1}$.
\begin{proposition} \label{gauss-eq}
The Gau\ss\ equation of $M$ in $\ep I\times_a\mathbb{M}^n_k( c )$ is
\begin{gather*}
R(X,Y,Z,W)=
 \left( \ep\frac{(a')^2}{a^2}-\frac{\ep_{0}}{a^2}  \right)\Big(
 \langle X,Z\rangle\langle Y,W\rangle -\langle Y,Z\rangle\langle X,W\rangle \Big)  
\\
+\left(\frac{a''}{a} -\frac{(a')^2}{a^2}+\frac{\ep\ep_{0}}{a^2}\right)\Big(
\langle X,Z\rangle\langle Y,T \rangle\langle W,T \rangle
-\langle Y,Z\rangle\langle X,T \rangle\langle W,T \rangle \\ 
-\langle X,W\rangle\langle Y,\dt \rangle\langle Z,T \rangle
+ \langle Y,W\rangle\langle X,T \rangle\langle Z,T \rangle\Big)  
\\ +\ep_{n+1}\Big( \produ{AY,Z}\produ{AX,W}-\produ{AY,W}\produ{AX,Z}\Big),
\end{gather*}
for any $X,Y,Z,W\in TM$.
\end{proposition}
\begin{proof}
We resort to (\ref{eq-Gauss}), Proposition \ref{barR} and Lemma \ref{lemma-T}.
\end{proof}
\begin{proposition} The Codazzi equation of $M$ in $\mmm$ is given by
\begin{equation}\label{eq-codazzi}
(\nabla_XA) Y - (\nabla_YA) X =
\ep_{n+1} f \left(\frac{a''}{a} -\frac{(a')^2}{a^2}+\frac{\ep\ep_{0}}{a^2}\right)
\Big(\produ{Y,T}X -\produ{X,T}Y\Big),
\end{equation}
for any $X, Y$ tangent to $M$.
\end{proposition}
\begin{proof}
By (\ref{eq-codazzigeneral}), we have to compute
$\bar{R}(X,Y,Z,e_{n+1})$ for any tangent vectors $X,Y,Z$ to $M$.  To do so, we recall Proposition \ref{barR}. Thus,
\begin{align*}
&\bar{R}(X,Y,Z,e_{n+1}) =
 \left( \ep\frac{(a')^2}{a^2}-\frac{\ep_{0}}{a^2}  \right)\Big(
 \langle X,Z\rangle\langle Y,{e}_{n+1}\rangle -\langle Y,Z\rangle\langle X,{e}_{n+1}\rangle \Big) \nonumber \\
 & +\left(\frac{a''}{a} -\frac{(a')^2}{a^2}+\frac{\ep\ep_{0}}{a^2}\right)\Big(
\langle X,Z\rangle\langle Y,\dt \rangle\langle {e}_{n+1},\dt \rangle
-\langle Y,Z\rangle\langle X,\dt \rangle\langle {e}_{n+1},\dt \rangle\\
&- \langle X,{e}_{n+1}\rangle\langle Y,\dt \rangle\langle Z,\dt \rangle
+ \langle Y,{e}_{n+1}\rangle\langle X,\dt \rangle\langle Z,\dt \rangle
\Big) \\
&= \ep_{n+1}f \left(\frac{a''}{a} -\frac{(a')^2}{a^2}+\frac{\ep\ep_{0}}{a^2}\right)
\Big(
\produ{Y,T}\produ{X,Z}-\produ{X,T}\produ{Y,Z}\Big).
\end{align*}
This yields the result. \end{proof}

\begin{lemma}\label{nablaT} The following equations hold for any $X$ tangent to $M$:
\begin{enumerate}
\item $\bar{\nabla}_X\dt = \frac{a'}{a}(X-\ep\produ{X,T}\dt)$.
\item $\nabla_XT=\frac{a'}{a}(X-\ep\produ{X,T}T)+fAX$.
\item $X(f)=-\ep_{n+1}\produ{AT,X}-\frac{a'}{a}\ep\produ{X,T}f$.
\end{enumerate}
\end{lemma}
\begin{proof}
First, we recall that $X=\ti{X}+\ep\eta(X)\dt$. Therefore, $\bar{\nabla}_X\dt = \bar{\nabla}_{\ti{X}}\dt = \frac{a'}{a}\ti{X}=\frac{a'}{a}(X-\ep\produ{X,T}\dt)$.
Next, we compute
$\nab_XT=\nab_X(\dt- e_{n+1})=\nab_X\dt-X(f)e_{n+1}-f\nab_Xe_{n+1}
= \frac{a'}{a}(X-\ep\produ{X,T}\dt) -X(f)e_{n+1}+fAX
= \frac{a'}{a}(X-\ep\produ{X,T}(T+fe_{n+1})) -X(f)e_{n+1}+fAX
= \frac{a'}{a}(X-\ep\produ{X,T}T)-\frac{\ep f a'}{a}e_{n+1})) -X(f)e_{n+1}+fAX.$
Now, each equation is just the tangential and the normal part
of $\nab_XT=\na_XT+\ep_{n+1}\produ{AX,T}e_{n+1}$.
\end{proof}

\section{Moving frames}

Elie Cartan developed the \textit{moving frame} technique. Definitions, basic results and some other details can be found in \cite[p. 18]{IL}. 
We will use the following convention on the ranges of indices, unless mentioned
otherwise:
\[ 1\leq i,j,k,l\leq n;\quad
1\leq u,v,w,\ldots\leq n+1;\quad
0\leq \alpha,\beta,\gamma, \ldots\leq n+1.\]
We recall that $M$ is a hypersurface of $\bar{P}^{n+1}$, hence $n=\dim M$.  Let
$(e_0,e_1,\dots,e_n,e_{n+1})$ be a local orthonormal frame on $M$, such that
$e_1,\ldots,e_n$ are tangent to $M$ and $e_{n+1}$ is normal to $M$ in
$\bar{P}^{n+1}$, with $\varepsilon_{\alpha}=\produ{e_{\alpha},e_{\alpha}}=\pm
1$. We define the matrix $G=(\varepsilon_{\alpha}\delta_{\alpha\beta})$. Let $(\omega_0,\ldots,\omega_{n+1})$ be the dual basis of $e_{\alpha}$, i.e.
$\omega_{\alpha}(e_{\beta})=\delta_{\alpha\beta}$, where
$\omega_{r}|_{TM}=0$, $r\in\{0,n+1\}$.  The dual 1-forms $\omega_{\alpha}$ can be obtained as
$\omega_{\alpha}(X) = \varepsilon_{\alpha}\langle e_{\alpha},X\rangle$.

Given  $(E_0,\dots,E_{n})$ a parallel orthonormal frame of $\mathbb{E}^{n+1}$, we construct
$(\E_0,\dots,\E_{n+1})=(\frac{E_0}{a},\dots\frac{E_{n}}{ a},\dt)$,
which  is an orthonormal frame of $\varepsilon I\times_a\mathbb{E}^{n+1}$.
If necessary, we reorder the basis $(\E_0,\ldots,\E_{n+1})$ to obtain
\[
\produ{\E_{\alpha},\E_{\alpha}}=\varepsilon_{\alpha}, \quad
\E_{n+1}=\dt.
\]
Next, we define the functions
$B_{\alpha\beta}:=\langle \E_{\alpha},e_{\beta}\rangle$ and the matrix $B=(B_{\alpha\beta})$.
We have:
\[ \sum_{\mu}\varepsilon_{\mu}B_{\mu\alpha}B_{\mu  \beta } =
\sum_{\mu}\varepsilon_{\mu}\langle \E_{\mu},e_{\alpha}\rangle\,\langle
\E_{\mu},e_{\beta }\rangle = \langle e_{\alpha},e_{\beta}\rangle
=\varepsilon_{ \alpha }\delta_{\alpha\beta }.\]
This equation reduces to $B^t GB=G$, which implies
$B^{-1}=GB^tG$, where $B^t$ is the transpose of $B$ and
$B^{-1}=(B^{\alpha\beta})$. Next, from the fact that $B^t GB=G$, we define the sets
\begin{gather*}
\mathbf{S} = \{ Z\in \mathcal{M}_{n+2}(\R) | Z^t G Z = G,\ \det Z=1\},\\
\mathfrak{s} = \{ H\in \mathcal{M}_{n+2}(\R) | H^tG+GH=0\}.
\end{gather*}
The set $\S$ is the connected component of the identity matrix, and is hence 
isometric to the Lie group $O^{+\uparrow}(n+2,q)$, where the
index of the metric is $q= k+\frac{\vert c-1\vert}{2}+\frac{\vert\varepsilon
-1\vert}{2}$. Clearly, $\mathfrak{s}$ is the Lie algebra associated with
$\mathbf{S}$. In other words, we have constructed a map $B:M\rightarrow \mathbf{S}$, and therefore, we immediately obtain the  $\mathfrak{s}$-valued 1-form
$B^{-1}dB$ on $M$. Let us now define the connection 1-forms $\Omega=(\omega_{\alpha\beta})$, 
\[
\omega_{\alpha\beta}(X)=\ep_{\alpha}\produ{e_{\alpha},\widetilde{\nabla}_X e_{\beta}}, \quad \mathrm{for\ any}\ X\in TM.
\]
The matrix $\Omega$ satisfies $ \Omega^tG+G\Omega =0$, or equivalently, 
$\omega_{\beta\alpha}=-\ep_{\alpha}\ep_{\beta}\omega_{\alpha\beta}.$ In particular,
\begin{equation}\label{connection-1-f}
\nabla e_i = \sum_k\omega_{ki} e_k, \quad
\bar{\nabla} e_{u} = \sum_{v}\omega_{vu} e_{v}, \quad
\widetilde{\nabla} e_{\alpha} = \sum_{\gamma} \omega_{\gamma\alpha} e_{\gamma},
\end{equation}
We now define the 1-form $\eta(X)=\langle T,X\rangle$ and functions
$T_{k}=\produ{e_{k},T}$, $T_{n+1}=\ep_{n+1} f$ and $T_{0}=0$. Clearly,
$\sum_kT_k\omega_k=\eta$. Obviously, we can recover the vectors
$e_{\beta}=\sum_{\gamma}\varepsilon_{\gamma}B_{\gamma\beta}\E_{\gamma}$.
Consequently, by (\ref{connection-1-f}), 
\begin{gather*}
\widetilde{\nabla}_{e_\alpha}e_\beta = \sum_{\mu}\omega_{\mu\beta}(e_{\alpha})e_{\mu}=\sum_{\gamma}\varepsilon_{\gamma}\Big(\sum_{\mu}
\omega_{\mu\beta}(e_\alpha)B_{\gamma\mu}\Big)\E_{\gamma} \\
=
\widetilde{\nabla}_{e_{\alpha}}\Big(\sum_{\gamma} \ep_{\gamma} B_{\gamma\beta}\E_{\gamma}\Big)
=\sum_{\gamma}\ep_{\gamma} dB_{\gamma\beta}(e_{\alpha})\E_{\gamma}
+ \sum_{\mu,\gamma} \varepsilon_{\gamma}\varepsilon_{\mu}
B_{\mu\alpha}
   B_{\gamma\beta}\widetilde{\nabla}_{\E_{\mu}}\E_{\gamma}.
\end{gather*}
By now, we just care for the last summand. To do so,
\begin{eqnarray*}
\widetilde{\nabla}_{\E_{n+1}}\E_{n+1}&=&\widetilde{\nabla}_{\dt}\dt=0,\quad 
\widetilde{\nabla}_{\E_{u}}\E_{n+1}
=\frac{1}{a}\widetilde{\nabla}_{E_{u}}\dt
=\frac{a'}{a}\E_{u},\\
\widetilde{\nabla}_{\E_{n+1}}\E_{u}&=&
-\frac{a'}{a^2}E_{u}+\frac{1}{a}\widetilde{\nabla}_{\dt}E_{u}=
-\frac{a'}{a^2}E_{u}+\frac{a'}{a^2}E_{u}=0, \\
\widetilde{\nabla}_{\E_{u}}\E_{v} &=&
\frac{1}{a^2}\widetilde{\nabla}_{E_{u}}E_{v}=
\frac{1}{a^2}\nabla^o_{E_{u}}E_{v}
-\frac{\langle \E_{u},\E_{v}\rangle}{a} \textrm{grad}(a) =-\frac{\varepsilon_{u}\delta_{u v}\varepsilon a'}{a}\dt.
\end{eqnarray*}
Consequently, by using the fact that the terms for $\mu=n+1$ vanish, we have
\begin{eqnarray*}
&&\widetilde{\nabla}_{e_\alpha}e_\beta - \sum_{\gamma}\ep_{\gamma} dB_{\gamma\beta}(e_{\alpha})\E_{\gamma} =
\sum_{\mu,\gamma} \varepsilon_{\gamma}\varepsilon_{\mu}B_{\mu\alpha}
B_{\gamma\beta}\widetilde{\nabla}_{\E_{\mu}}\E_{\gamma}
\\ && =
\sum_{u,v}\varepsilon_{v}\varepsilon_{u}B_{u\alpha}B_{v\beta}
\widetilde{\nabla}_{\E_{u}}\E_{v}
+\sum_{u}\varepsilon_{0}\varepsilon_{u}B_{u\alpha} B_{0 \beta}\widetilde{\nabla}_{\E_{u}}\E_{0}
\\&&
=
\ep_{0}\frac{a'}{a}\sum_{v}\varepsilon_{v}B_{v\alpha }B_{0 \beta}\E_{v} 
-\frac{\ep_{0} a'}{a}\sum_{u}\varepsilon_{u}B_{u\alpha}B_{u\beta}\E_{0}.
\end{eqnarray*}
By comparing coordinates, we get for $\gamma=n+1$
\begin{equation}\label{eq_movingframes1}
\sum_{\mu}B_{n+1\,\mu}\omega_{\mu\beta}(e_\alpha)=
dB_{n+1\,\beta}(e_{\alpha})-\frac{a'}{a}\sum_{u}\varepsilon_{u}B_{u\alpha}B_{u\beta}.
\end{equation}
and for $\gamma=0,\ldots,n$,
\begin{equation}\label{eq_movingframes2}
\sum_{\mu}B_{\gamma\mu}\omega_{\mu\beta}(e_\alpha)=
 dB_{\gamma\beta}(e_{\alpha})
   +\frac{\varepsilon a'}{a} B_{\gamma\alpha}B_{n+1\,\beta}.
\end{equation}
Using the fact that $B_{\mu\alpha}=\sum_{\gamma}B_{\mu\gamma}\omega_{ \gamma }(e_{\alpha})$, we get for equation \eqref{eq_movingframes1}, 
$\sum_{\mu}B_{n+1\,\mu}\omega_{\mu\beta}
-dB_{n+1\,\beta}=-\frac{a'}{a}\sum_{\gamma}\sum_{u}\varepsilon_{u}B_{u\beta}B_{u\gamma }
\omega_{\gamma }=
-\frac{a'}{a}\sum_{\gamma,\mu}\varepsilon_{\mu}B_{\mu\beta}B_{\mu \gamma }\omega_{\gamma}
+\varepsilon\frac{a'}{a}\sum_{\gamma}B_{n+1\,\beta}B_{n+1\,\gamma}\omega_{\gamma}$, 
and for equation \eqref{eq_movingframes2}, 
$\sum_{\mu}B_{\gamma\mu}\omega_{\mu\beta}=
dB_{\gamma\beta}+\frac{\varepsilon a'}{a}\sum_{\kappa}B_{n+1\, \beta} B_{\gamma\kappa}\omega_{ \kappa }$, for any $\gamma=0,\ldots,n+1.$ 
Finally, for all $\gamma=\alpha$
\[
\sum_{\mu}B_{\alpha\mu}\omega_{\mu\beta}  = 
dB_{\alpha\beta}
+\frac{\varepsilon a'}{a}B_{n+1\,\beta}\sum_{\gamma}B_{\alpha\gamma }\omega_{ \gamma}
- \frac{a'}{a}\ep_{\beta}\delta_{\alpha\,n+1}\omega_{\beta}.
\] 
Moreover, we have
$\sum_{\mu}B^{\alpha\mu}\delta_{\mu 0}\ep_{\beta}\omega_{\beta}
=B^{\alpha 0}\ep_{\beta}\omega_{\beta}
=\ep_{\beta}\ep_{\alpha}\ep B_{0 \alpha}\omega_{\beta}$, that is,
\begin{equation}\label{allalpha}
\omega_{\alpha\beta} = \sum_{\mu}B^{\alpha\mu}dB_{\mu\beta}+
\frac{\ep a'}{a}\Big(B_{n+1\,\beta}\omega_{\alpha}
-\ep_{\beta}\ep_{\alpha}B_{n+1\,\alpha}\omega_{\beta}\Big).
\end{equation}
Finally, we obtain
\begin{equation}\label{MaurerCartan}
 \Omega -\X= B^{-1}dB, \quad \X_{\alpha\beta} =
\frac{\ep a'}{a}\Big(B_{n+1\,\beta}\omega_{\alpha}
-\ep_{\beta}\ep_{\alpha}B_{n+1\,\alpha}\omega_{\beta}\Big). 
\end{equation}
We point out that $B_{n+1\,\alpha} =\produ{\E_{n+1},e_{\alpha}}=\produ{\dt,e_{\alpha}}=T_{\alpha}$.  

\section{Main Theorem}

Let $(M,\produ{,})$ be a  semi-Riemannian manifold with its Levi-Civita connection $\na$, its Riemann tensor $R$.   We choose numbers $\ep,\ep_0,\ep_{n+1}\in\{-1,1\}$   and $c=\ep_{0}$,
and smooth functions $a:I\subset\R\rightarrow\R^{+}$, $T_{n+1}:M\rightarrow\R$ and $\pi:M\rightarrow I$.  We construct the vector field $T\in\mathfrak{X}(M)$ by $T=\ep\, \grad(\pi)$,  with its 1-form $\eta(X)=\produ{X,T}$. Also, consider a  tensor $A$ of type (1,1) on $M$. 
\begin{definition}\label{structureconditions} \rm Under the previous conditions, we will say that $M$ satisfies the \emph{structure conditions} if the following conditions hold:
\begin{enumerate}[(A)]
\item $A$ is $\produ{,}$-self adjoint;
\item\label{condA} $\ep = \produ{T,T}+\ep_{n+1}T_{n+1}^2$;
\item\label{condB} $\nabla_XT=\frac{a'\circ\pi}{a\circ\pi}(X-\ep\eta(X)T)+\ep_{n+1}T_{n+1}AX$, for any $X\in TM$;
\item\label{condC} $X(T_{n+1})=-\produ{AT,X}-\ep\frac{a'\circ\pi}{a\circ\pi} T_{n+1}\eta(X)$, for any $X\in TM$;

\item  \label{eq-codazziM} \textit{Codazzi equation:} for any $X, Y\in TM$, it holds
\[(\nabla_XA) Y - (\nabla_YA) X =
T_{n+1}  \left(\frac{a''\circ\pi}{a\circ\pi} -\frac{(a'\circ\pi)^2}{(a\circ\pi)^2}+\frac{\ep\ep_{0}}{(a\circ\pi)^2}\right) \Big(\eta(Y)X -\eta(X)Y\Big);
\]

\item \textit{Gau\ss\, equation:} \label{eq-GaussM} for any $X,Y,Z,W\in TM$, it holds 
\begin{gather*}
R(X,Y,Z,W)=
 \left( \ep\frac{(a'\circ\pi)^2}{(a\circ\pi)^2}-\frac{\ep_{0}}{(a\circ\pi)^2}  \right)
 \Big( \produ{X,Z}\produ{Y,W}-\produ{Y,Z}\produ{X,W}\Big) \nonumber \\
\quad  +\left(\frac{a''\circ\pi}{a\circ\pi} -\frac{(a'\circ\pi)^2}{(a\circ\pi)^2}+\frac{\ep\,\ep_{0}}{(a\circ\pi)^2}\right)
 \Big(\produ{X,Z}\eta(Y)\eta(W) -\produ{Y,Z}\eta(X)\eta(W)
\\ 
- \produ{X,W}\eta(Y)\eta(Z)+ \produ{Y,W}\eta(X)\eta(Z)\Big)
 +\ep_{n+1}\Big( \produ{AY,Z}\produ{AX,W}-\produ{AY,W}\produ{AX,Z}\Big).
\end{gather*}
\end{enumerate}
\end{definition}
We recall the warped product $(\bar{P}^{n+1}=I\times \mathbb{M}_k^n(c),\produ{,}_1=\ep dt^2+a^2g_o)$.

\begin{theorem} \label{main} Let $(M,\produ{,})$ a semi-Riemannian manifold satisfying the structure conditions. Then, for each point $p\in M$, there exists a neighborhood $\mathcal U$ of p on $M$, a metric immersion $\chi:(\mathcal {U},\produ{,})\rightarrow (\bar{P}^{n+1},\produ{,}_1)$ and a normal unit vector field $e_{n+1}$ along $\chi$ such that:
\begin{enumerate}
\item $\ep_{n+1}=\produ{e_{n+1},e_{n+1}}_1$;
\item $\pi_I\circ\chi=\pi$, where $\pi_I:\mmm\rightarrow I$ is the projection;
\item The shape operator associated with $e_{n+1}$ is $A$;
\item (\ref{eq-codazziM}) is the Codazzi equation, and  (\ref{eq-GaussM}) is the Gau\ss\, equation,
\item  \label{dt} and along $\chi$, it holds
$\dt = T+\ep_{n+1}T_{n+1}e_{n+1}$.
\end{enumerate}
\end{theorem}
\begin{proof}
Given a point $x\in M$, around it we consider a local orthonormal frame $\{e_1,\ldots,e_n\}$ on $M$, with their signs $\ep_i=g(e_i,e_i)=\pm 1$, and its corresponding dual basis of 1-forms $\{\omega_1,\ldots,\omega_n\}$. We point out that an alternative definition for these 1-forms is
$\omega_i(X)=\ep_i\produ{e_i,X}$, for any $X\in TM$. We also need to define $\omega_{n+1}=\omega_{0}=0.$ With the help of the tensor $SY = -\big(Y-\ep \eta(Y)T\big)/(ac)$, for any $Y\in TM$, we construct the following 1-forms
\begin{align}
&\omega_{ij}(X)=\ep_i \produ{e_i,\nabla_Xe_j}, \quad \omega_{i\,n+1}(X)=-\ep_i \produ{e_i,AX},\nonumber \\
&\omega_{i 0}(X)=-\ep_i\produ{e_i,SX},  \quad
\omega_{n+1,0}=-\frac{\ep\ep_{n+1}}{c (a\circ\pi)}T_{n+1}\eta,
\label{connforms} 
\quad 
\omega_{\alpha\beta}=-\ep_{\alpha}\ep_{\beta}\omega_{\beta\alpha}, 
\end{align}
for any $X\in TM$, known as the \textit{connection 1-forms}. In this way, we consider the $\mathfrak{s}$-valued matrix $\Omega=(\omega_{\alpha\beta})$. As a consequence, we get $\nabla e_i = \sum_k\omega_{ki}e_k.$ Now, we define the functions $T_i=\eta(e_i)$, $i\in\{1,\ldots,n\}$, $T_{0}=0$. We point out that by condition (\ref{condA}), we have $\ep=\sum_{\gamma}\ep_{\gamma}T_{\gamma}^2.$ Next,  we also construct the matrices $\X=(\X_{\alpha\beta})$ and $\Upsilon$ as
\begin{equation} \label{refini}
\X_{\alpha\beta} = \frac{\ep a'}{a}\big(T_{\beta}\omega_{\alpha}-\ep_{\alpha}\ep_{\beta}T_{\alpha}\omega_{\beta}\big), \quad \Upsilon = \Omega-\X.
\end{equation}
A simple computation shows
\[
d\Upsilon+\Upsilon\wedge\Upsilon = d\Omega-d\X+\Omega\wedge\Omega-\Omega\wedge \X-\X\wedge\Omega+\X\wedge\X.
\]
Thus, our target consist of proving that the second half of this equality vanishes. Since the
computation is rather lengthy, we will split it in some lemmata.
\begin{lemma} \label{deta} $d\eta=0$.
\end{lemma}
\begin{proof} 
Since $T=\ep\,\grad(\pi)$, we obtain that $\eta= \ep d\pi$. Therefore, $d\eta =0$.
\end{proof}
We define  the matrices $\varpi=(\omega_{\alpha})$ and
$\Gamma=(\Gamma_{\alpha\beta})=d\Omega+\Omega\wedge\Omega$.
\begin{lemma}\label{Cartan-first}
\begin{align*}
d\varpi&=-\Omega\wedge\varpi, \quad
\Gamma_{\alpha\beta} =-\ep_{\alpha}\ep_{\beta}\Gamma_{\beta\alpha},
\\
\Gamma_{ij} &= \ep\ep_j \frac{(a')^2}{a^2} \omega_j\wedge\omega_i
 -\left(\frac{a''}{a} -\frac{(a')^2}{a^2}\right)
\big( T_j\omega_i-\ep_i\ep_jT_i\omega_j\big)\wedge\eta,\\
\Gamma_{i\,n+1}&=
T_{n+1}\Big(\frac{a''}{a}-\frac{(a')^2}{a^2}\Big)\eta\wedge\omega_i,\quad 
\Gamma_{u 0}=0,
\end{align*}
\end{lemma}
\begin{proof} Given $X,Y\in TM$, since $\omega_{n+1}=\omega_{0}=0$, we have
\begin{align*}
& d\omega_{i}(X,Y)  = X(\omega_{i}(Y))-Y(\omega_{i}(X))-\omega_{i}([X,Y]) \\ &
=\ep_{i} \produ{\nabla_Xe_{i},Y}+\ep_{i}\produ{e_{i},\nabla_XY}
- \ep_{i} \produ{\nabla_Ye_{i},X}-\ep_{i}\produ{e_{i},\nabla_YX}
-\ep_{i}\produ{e_{i},[X,Y]}
\\
&=\ep_{i}\sum_{k} \omega_{k i} (X)\produ{e_{k},Y}
-\ep_{i}\sum_{k} \omega_{k i} (Y)\produ{e_{k},X}
=-\sum_{\gamma} \omega_{i\gamma}\wedge\omega_{\gamma} (X,Y).
\end{align*}
On the other hand, by (\ref{connforms}),
\begin{align*}
&\sum_{\gamma}\omega_{n+1\,\gamma}\wedge\omega_{\gamma}(X,Y)=
\sum_{k}\omega_{n+1\,k}\wedge\omega_{k}(X,Y)\\
&= \sum_k\big(\ep_{n+1}\produ{e_k,AX}\ep_k\produ{e_k,Y}
- \ep_{n+1}\produ{e_k,AY}\ep_k\produ{e_k,X}\big)
\\& =\ep_{n+1}\big(\produ{Y,AX}-\produ{X,AY}\big)=0=-d\omega_{n+1}(X,Y).
\end{align*}
Also,
\begin{align*}
&\sum_{\gamma}\omega_{0 \gamma}\wedge\omega_{\gamma}(X,Y) =
\sum_{k}\omega_{0 k}\wedge\omega_{k}(X,Y) \\
& = \frac{1}{ac}\sum_k \Big( -\ep_{0}\big( \produ{e_k,X}-\ep T_k\eta(X)\big)\omega_k(Y)
+\ep_{0}\big( \produ{e_k,Y}-\ep T_k\eta(Y)\big)\omega_k(X)\Big)
\\ & = \frac{\ep_{0}}{ac}\Big( -\produ{Y,X}+\ep\eta(Y)\eta(X)+\produ{X,Y}-\ep\eta(X)\eta(Y)\Big) =0
= -d\omega_{0}(X,Y).
\end{align*}
Next, given $X,Y\in TM$, we compute
\begin{align*}
d\omega_{i j }(X,Y)&=X(\omega_{i j }(Y))
-Y(\omega_{i j }(X))-\omega_{i j }([X,Y]) \\
&=\ep_{i }X(\produ{e_{i },\na_{Y}e_{j }})
-\ep_{i }Y(\produ{e_{i },\na_{X}e_{j }})
-\ep_{i }\produ{e_{i },\na_{[X,Y]}e_{j }} \\
&=\ep_{i }\produ{\na_{X}e_{i },\na_{Y}e_{j }}
- \ep_{i }\produ{\na_{Y}e_{i },\na_{X}e_{j }}
+\ep_{i }{R}(X,Y,e_{j },e_{i }).
\end{align*}
On the other hand, by (\ref{connection-1-f}), 
$\produ{\na_{X}e_{i },\na_{Y}e_{j }} =
\sum_{ k }\ep_{ k } \omega_{ k  i }(X)\omega_{ k  j }(Y) =
-\sum_{ k }\ep_{i }\omega_{i  k }(X)\omega_{ k  j }(Y).$ Therefore, 
\[d\omega_{i j }(X,Y) = -\sum_k \omega_{ik}\wedge\omega_{kj}(X,Y)-\ep_{i} R(X,Y,e_i,e_j), \]
which implies
\begin{align*}
& d\omega_{i j }(X,Y) +\sum_{\gamma} \omega_{i \gamma}\wedge\omega_{\gamma j }(X,Y)\\
&=\omega_{i0}\wedge\omega_{0j}(X,Y)+\omega_{i\,n+1}\wedge\omega_{n+1,j}(X,Y)
+\ep_{i } {R}(X,Y,e_{j},e_{i}) \\
& = -\ep_i\ep_{n+1}\produ{e_i,AX}\produ{e_j,AY} + \ep_i\ep_{n+1}\produ{e_i,AY}\produ{e_j,AX} 
 -\ep_i\ep_{0}\produ{e_i,SX}\produ{e_j,SY} \\
&\quad + \ep_i\ep_{0}\produ{e_i,SY}\produ{e_j,SX} 
+\ep_i\Bigg(\Big( \ep\frac{(a')^2}{a^2}-\frac{\ep_{0}}{a^2}\Big)\Big(
 \produ{X,e_j}\produ{Y,e_i}-\produ{Y,e_j}\produ{X,e_i}\Big)  \\
&\quad  +\left( \frac{\ep\ep_0 }{a^2 } +\frac{a''}{a} -\frac{ (a')^2 }{ a^2 } \right)
\Big(\produ{X,e_j}\eta(Y)\eta(e_i) -\produ{Y,e_j}\eta(X)\eta(e_i) \\
 &\quad
- \produ{X,e_i}\eta(Y)\eta(e_j)+ \produ{Y,e_i}\eta(X)\eta(e_j)\Big)
\\ &\quad
 +\ep_{n+1}\Big( \produ{AY,e_j}\produ{AX,e_i}-\produ{AY,e_i}\produ{AX,e_j}\Big)
  +\ep_{0}\Big( \produ{SY,e_j}\produ{SX,e_i}-\produ{SY,e_i}\produ{SX,e_j}\Big) \Bigg)
  \\
& =  \ep\ep_j \frac{(a')^2}{a^2} \omega_j\wedge\omega_i(X,Y)
 -\left(\frac{a''}{a} -\frac{(a')^2}{a^2}\right)
\big( T_j\omega_i-\ep_i\ep_jT_i\omega_j\big)\wedge\eta (X,Y).
\end{align*}
Next, given $X,Y\in TM$, we compute
\begin{align*}
& d\omega_{i n+1 }(X,Y)=X(\omega_{i n+1 }(Y))
-Y(\omega_{i n+1 }(X))-\omega_{i n+1 }([X,Y]) \\
&=-\ep_i\produ{\nabla_Xe_i,AY} +\ep_i\produ{\nabla_Ye_i,AX}
+\ep_i\produ{e_i,(\nabla_YA)X-(\nabla_XA)Y}\\
&=-\ep_i\sum_k\omega_{ki}(X)\produ{e_k,AY} + 
\ep_i\sum_k\omega_{ki}(Y)\produ{e_k,AX}+\ep_i\produ{e_i,(\nabla_YA)X-(\nabla_XA)Y}\\
&=\sum_k\Big( \omega_{k\,n+1}(X)\omega_{ik}(Y)- \omega_{k\,n+1}(Y)\omega_{ik}(X)\Big)
+\ep_i\produ{e_i,(\nabla_YA)X-(\nabla_XA)Y}\\
&=-\sum_{\gamma}\omega_{i\gamma}\wedge\omega_{\gamma\,n+1}(X,Y)+
T_{n+1}\Big(\frac{a''}{a}-\frac{(a')^2}{a^2}\Big)\eta\wedge\omega_i(Y,X).
\end{align*}
The next case is 
\begin{align*}
d\omega_{i0}(X,Y)&=X(\omega_{i0}(Y))-Y(\omega_{i0}(X))-\omega_{i0}([X,Y])\\
&=\ep_i\big( \produ{\nabla_Ye_i,SX}- \produ{\nabla_Xe_i,SY}
-\produ{e_i,(\nabla_YS)X-(\nabla_XS)Y}\big).
\end{align*}
On one hand, for any $U\in\mathfrak{X}(M)$, it holds $\nabla_Y \big(\frac{-1}{ac} U \big) = \frac{\ep a'}{ca^2}\eta(Y) U -\frac{1}{ac}\nabla_YU$, so that
\begin{align*}
&(\nabla_YS)X-(\nabla_XS)Y = \nabla_YSX-S\nabla_YX-\nabla_XSY+S\nabla_XY \\
&= \frac{\ep a'}{ca^2}\eta(Y) \big(X-\ep\eta(X)T\big) 
-\frac{1}{ac}\nabla_Y(X-\ep\eta(X)T)+\frac{1}{ac}\nabla_YX-\frac{\ep}{ac}\eta(\na_YX)T\\
&\quad -\frac{\ep a'}{ca^2}\eta(X) \big(Y-\ep\eta(Y)T\big) 
+\frac{1}{ac}\nabla_X(Y-\ep\eta(Y)T)+\frac{1}{ac}\nabla_XY-\frac{\ep}{ac}\eta(\na_XY)T\\
&= \frac{\ep a'}{ca^2}\big(\eta(Y)X-\eta(X)Y\big) +\frac{\ep}{ac}\Big(\eta(X)\nabla_YT-\eta(Y)\nabla_XT\Big) \\
&= \frac{\ep \ep_{n+1}T_{n+1} }{a c} (\eta(X)AY-\eta(Y)AX).
\end{align*}
Therefore, we have
\begin{align*}
&d\omega_{i0}(X,Y)=
\ep_i\Big( \produ{\nabla_Ye_i,SX}- \produ{\nabla_Xe_i,SY}
+\frac{\ep\ep_{n+1}T_{n+1}}{ac}\produ{e_i,\eta(X)AY-\eta(Y)AX}\Big)\\
&=\ep_i\Big( \sum_{k}\big( \omega_{ki}(Y)\produ{e_k,SX}- \omega_{ki}(X)\produ{e_k,SY}\big) 
+\frac{\ep\ep_{n+1}T_{n+1}}{ac}\produ{e_i,\eta(X)AY-\eta(Y)AX}\Big)\\
&=-\sum_{\gamma}\omega_{i\gamma}\wedge\omega_{\gamma 0}(X,Y)
+\omega_{i\,n+1}\wedge\omega_{n+1,0}(X,Y)-\frac{\ep\ep_{n+1}T_{n+1}}{ac}\eta\wedge\omega_{i\,n+1}(X,Y)\\
&=-\sum_{\gamma}\omega_{i\gamma}\wedge\omega_{\gamma 0}(X,Y).
\end{align*}
Next, we easily see $d\omega_{n+1,0}=-\frac{\ep\ep_{n+1}}{ac}dT_{n+1}\wedge\eta$. Therefore,
\begin{align*}
&d\omega_{n+1,0}(X,Y)=-\frac{\ep\ep_{n+1}}{ac}dT_{n+1}\wedge\eta(X,Y) 
= -\frac{\ep\ep_{n+1}}{ac}dT_{n+1}\Big( X(T_{n+1})\eta(Y)-Y(T_{n+1})\eta(X)\Big)\\
&
=\ep_{n+1} \sum_k \ep_k\big(\produ{AX,e_k}\produ{SY,e_k}-\produ{AY,e_k}\produ{SX,e_k}\big)\\
&=-\sum_k \omega_{n+1\,k}\wedge\omega_{k 0}(X,Y)
=-\sum_{\gamma} \omega_{n+1\,\gamma}\wedge\omega_{\gamma 0}(X,Y). \qedhere
\end{align*}
\end{proof}

\begin{lemma}
\begin{eqnarray*}
d\X_{\alpha\beta}&=&-\ep_{\alpha}\ep_{\beta}d\X_{\beta\alpha}, \quad 
d\X_{u 0}=0, \\
d\X_{ij}&=&
\frac{a''a-2(a')^2}{a^2}\eta\wedge(T_{j}\omega_{i}
-\ep_i\varepsilon_{j}T_{i}\omega_{j})+\frac{\varepsilon a'}{a}\Big(T_jd\omega_i
-\ep_i\ep_jT_id\omega_j\Big)
\\
&&+2\ep\ep_j \left(\frac{a'}{a}\right)^2\omega_j\wedge\omega_i
+\frac{\varepsilon
a'}{a}\sum_{u}T_u\Big(\omega_{uj}\wedge\omega_i-\ep_i\ep_j\omega_{ui}\wedge\omega_j \Big),
\\
d\X_{i\,n+1}&=&
\frac{a''a-2(a')^2}{a^2}T_{n+1}\eta  \wedge\omega_{i}
+\frac{\varepsilon a'}{a}\Big(\sum_kT_k\omega_{k\,n+1}\wedge \omega_{i} +T_{n+1}d\omega_{i}\Big).
\end{eqnarray*}
\end{lemma}
\begin{proof}
Since $\X_{\alpha\beta}=-\ep_{\alpha}\ep_{\beta}\X_{\beta\alpha}$, we trivially have
$d\X_{\alpha\beta}=-\ep_{\alpha}\ep_{\beta}d\X_{\beta\alpha}$.  Next,
\begin{align*}
d\X_{\alpha\beta}&=
d\left(\frac{\varepsilon a'}{a}\left(T_{\beta}\omega_{\alpha}
-\ep_{\alpha}\ep_{\beta}T_{\,\alpha}\omega_{\beta}\right)\right)\\
&=\Big(\sum_k e_k\left(\frac{\ep a'}{a}\right)\omega_k\Big)
\wedge(T_{\,\beta}\omega_{\alpha}-\varepsilon_{\alpha}\varepsilon_{\beta}T_{\,\alpha}\omega_{\beta}
)\\
&+\frac{\varepsilon a'}{a}\Big(dT_{\,\beta}\wedge\omega_{\alpha}
+T_{\,\beta}d\omega_{\alpha}
-\ep_{\alpha}\varepsilon_{\beta}dT_{\,\alpha}\wedge\omega_{\beta}
-\ep_{\alpha}\varepsilon_{\beta}T_{\,\alpha}d\omega_{\beta}\Big) \\
&= \frac{a''a-(a')^2}{a^2}\eta  \wedge(T_{\,\beta}\omega_{\alpha}
-\varepsilon_{\alpha}\varepsilon_{\beta}T_{\,\alpha}\omega_{\beta})
\\
&\quad +\frac{\varepsilon a'}{a}
\Big(dT_{\,\beta}\wedge\omega_{\alpha}
+T_{\,\beta}d\omega_{\alpha}
-\varepsilon_{\alpha}\varepsilon_{\beta}dT_{\,\alpha}\wedge\omega_{\beta}
-\ep_{\alpha}\varepsilon_{\beta}T_{\,\alpha}d\omega_{\beta}\Big).
\end{align*}
For $\beta=0$ and any $u>0$, since $T_{ 0}=0$ and $\omega_{0}=0$, then $d\X_{u 0}=0$. For $\beta=n+1$ and $i<n+1$, since  $\omega_{n+1}=0$, we have 
\begin{align*}
d\X_{i\,n+1}& = \frac{a''a-(a')^2}{a^2}T_{n+1}\eta  \wedge\omega_{i}
+\frac{\varepsilon a'}{a}\Big( dT_{n+1}\wedge\omega_{i}
+T_{n+1}d\omega_{i}\Big).
\end{align*}
By condition (\ref{condB}), and the fact $T=\sum_k\ep_kT_ke_k$, we see
$dT_{n+1}(X)=  -\produ{AX,T}-\ep\frac{a'}{a}T_{n+1}\eta(X)$ $
=-\sum_k \ep_kT_k\produ{e_k,AX}- \ep T_{n+1}\frac{a'}{a}\eta(X)$, 
and consequently 
$dT_{n+1}(X)=\sum_kT_k\omega_{k\,n+1}(X)-\ep T_{n+1}\frac{a'}{a}\eta(X)$.
With this,
\begin{align*}
d\X_{i\,n+1}
& = \frac{a''a-(a')^2}{a^2}T_{n+1}\eta  \wedge\omega_{i}
+\frac{\varepsilon a'}{a}\Big(
\sum_kT_k\omega_{k\,n+1}\wedge \omega_{i} -\ep\frac{a'}{a}T_{n+1}\eta\wedge\omega_i
+T_{n+1}d\omega_{i}\Big) \\
&= \frac{a''a-2(a')^2}{a^2}T_{n+1}\eta  \wedge\omega_{i}
+\frac{\varepsilon a'}{a}\Big(\sum_kT_k\omega_{k\,n+1}\wedge \omega_{i} +T_{n+1}d\omega_{i}\Big).
\end{align*}

Next, for $\alpha=i$ and $\beta=j$, we have
\begin{align*}
d\X_{ij} &= \frac{a''a-(a')^2}{a^2}\eta  \wedge(T_{j}\omega_{i}
-\varepsilon_{i}\varepsilon_{j}T_{i}\omega_{j})
\\
&+\frac{\varepsilon a'}{a} \Big(dT_{j}\wedge\omega_{i} +T_{j}d\omega_{i}
-\varepsilon_{i}\varepsilon_{j}dT_{i}\wedge\omega_{j} -\ep_{i}\varepsilon_{j}T_{i}d\omega_{j}\Big).
\end{align*}
We need the following computation
$dT_k=\sum_{l}e_{l}(\langle T,e_k\rangle)\omega_{l}
=\sum_{l}\big(\langle\nabla_{e_{l}}T,e_k\rangle\omega_{l}+\langle
T,\nabla_{e_{l}}e_k\rangle\big)\omega_{l}$ $=
\sum_l \Big( \produ{\frac{a'}{a}(e_{l}-\ep T_{l}T)+\ep_{n+1}T_{n+1}Ae_l,e_k }
+\sum_{j} \omega_{jk}(e_l)T_j \Big) \omega_l$, 
which implies 
\begin{equation}  \label{dT_k}
dT_k=\frac{a'}{a}\ep_k\omega_k -\ep\frac{a'}{a}T_k\eta
+\sum_{u}T_u\omega_{uk}. 
\end{equation}
In this way, a straight forward computation yields
\begin{align*}
d\X_{ij}
&=\frac{a''a-2(a')^2}{a^2}\eta\wedge(T_{j}\omega_{i}-\ep_i\varepsilon_{j}T_{i}\omega_{j})
+\frac{\varepsilon a'}{a}\Big(T_jd\omega_i -\ep_i\ep_jT_id\omega_j\Big) \\
&\quad
+2\frac{\ep\ep_j a'}{a}\omega_j\wedge\omega_i +\frac{\varepsilon 
a'}{a}\Bigg(\sum_{u}T_u(\omega_{uj}\wedge\omega_i-\ep_i\ep_j\omega_{ui}\wedge\omega_j)\Bigg).\qedhere
\end{align*}
\end{proof}

\begin{lemma}
\[(\X\wedge \X)_{\alpha\beta}=
\left(\frac{a'}{a}\right)^2
\Big(
\big(T_{\beta}\omega_{\alpha}
-\varepsilon_{\alpha}\varepsilon_{\beta}T_{\alpha}\omega_{\beta}
\big)\wedge\eta
-\ep\varepsilon_{\beta}\omega_{\alpha} \wedge\omega_{\beta}
\Big).\]
\end{lemma}
\begin{proof} We recall (\ref{condA}). Also, given $X\in TM$, we have $\sum_{\gamma}T_{\gamma}\omega_{\gamma}(X)=\sum_{k}T_{k}\omega_{k}(X)$ $=\eta(X)$. Then,
\begin{align*}
(\X\wedge \X)_{\alpha\beta}=&
\sum_{\gamma}\Big(\frac{\varepsilon a'}{a}
\left(T_{\gamma}\omega_{\alpha}
-\ep_{\alpha}\varepsilon_{\gamma}T_{\alpha}\omega_{\gamma}\right)
\wedge
\frac{\varepsilon a'}{a}\left(T_{\beta}\omega_{\gamma} -
\varepsilon_{\gamma}\varepsilon_{\beta}T_{\gamma}\omega_{\beta}\right)\Big)
\\
&=\left(\frac{a'}{a}\right)^2
\sum_{\gamma}\Big(T_{\gamma}T_{\beta}\omega_{\alpha}
\wedge\omega_{\gamma}-\varepsilon_{\beta}
\varepsilon_{\gamma}T_{\gamma}^2\omega_{\alpha} \wedge\omega_{\beta}
\\
&\quad
-\ep_{\alpha}\varepsilon_{\gamma}T_{\alpha}T_{\beta}\omega_{\gamma}\wedge\omega_{\gamma}
+\varepsilon_{\gamma}\varepsilon_{\gamma}\varepsilon_{\alpha}\varepsilon_{\beta}T_{\alpha}T_{\gamma}
\omega_{\gamma}\wedge\omega_{\beta}\Big)
\\
&= \left(\frac{a'}{a}\right)^2
\Big(
\big(T_{\beta}\omega_{\alpha}
-\varepsilon_{\alpha}\varepsilon_{\beta}T_{\alpha}\omega_{\beta}
\big)\wedge\eta
-\ep\varepsilon_{\beta}\omega_{\alpha} \wedge\omega_{\beta}
\Big).\qedhere
\end{align*}
\end{proof}
Now, we put $\Phi=\Omega\wedge\X+\X\wedge\Omega$.
\begin{lemma}
\begin{eqnarray}
\Phi_{\alpha\beta} &=&-\ep_{\alpha}\ep_{\beta}\Phi_{\beta\alpha}, \quad \Phi_{u 0}=0, 
\\
\Phi_{ij}&=&\frac{\varepsilon a'}{a}\Big(\ep_{i}\ep_{j }T_id\omega_{j } -T_jd\omega_{i }
+\sum_{u}T_{u}\big(\omega_{i }\wedge\omega_{u j}
-\varepsilon_{j }\varepsilon_u\omega_{i u}\wedge\omega_{j}\big)\Big) \\
\Phi_{i\,n+1} &=& \frac{\ep
a'}{a}\Big(-T_{n+1}d\omega_i+\sum_{k}T_k\omega_i\wedge\omega_{k\,n+1}\Big).
\end{eqnarray}
\end{lemma}
\begin{proof} By construction, $\X_{\alpha\beta}=-\ep_{\alpha}\ep_{\beta}\X_{\beta\alpha}$. This
shows $\Phi_{\alpha\beta}=-\ep_{\alpha}\ep_{\beta}\Phi_{\beta\alpha}$. In general, we get
\begin{align*}
\Phi_{\alpha\beta}&=\sum_{\gamma}\omega_{\alpha\gamma}\wedge\left(\frac{\varepsilon a'}{a}
\left(T_{\beta}\omega_{\gamma}-\varepsilon_{\beta}\varepsilon_{\gamma}T_{\gamma}\omega_{\beta}
\right)\right)
+\sum_{\gamma}\frac{\varepsilon a'}{a}
\left(T_{\gamma}\omega_{\alpha}
-\varepsilon_{\gamma}\varepsilon_{\alpha}T_{\alpha}\omega_{\gamma}\right)\wedge\omega_{\gamma\beta}
\\
&=
\frac{\varepsilon a'}{a}\sum_{\gamma}
\left(T_{\beta}\omega_{\alpha\gamma}\wedge\omega_{\gamma}
-\varepsilon_{\beta}\varepsilon_{\gamma}T_{\gamma}\omega_{\alpha\gamma}\wedge\omega_{\beta}
+T_{\gamma}\omega_{\alpha}\wedge\omega_{\gamma\beta}
-\varepsilon_{\alpha}\varepsilon_{\gamma}T_{\alpha}\omega_{\gamma}\wedge\omega_{\gamma\beta}
\right).
\end{align*}
By Lemma \ref{Cartan-first},  
\begin{align*}\Phi_{\alpha\beta}&
=\frac{\varepsilon a'}{a}
\Big(\ep_{\alpha}\ep_{\beta}T_{\alpha}d\omega_{\beta} -T_{\beta}d\omega_{\alpha}
 +\sum_{\gamma}\big(T_{\gamma}\omega_{\alpha}\wedge\omega_{\gamma\beta}
-\varepsilon_{\beta}\ep_{\gamma}T_{\gamma}\omega_{\alpha\gamma}\wedge\omega_{
\beta}\big)\Big).
\end{align*}
For the case $\alpha=i$, $\beta=n+1$, the result is immediate due to the fact $\omega_{n+1}=0$.
For the case $\alpha=i$, $\beta=0$, since $\omega_{0}=0$ and $T_{0}=0$, we begin by writing down
$\Phi_{i 0}=\frac{\varepsilon a'}{a}
\sum_{\gamma} T_{\gamma}\omega_{i}\wedge\omega_{\gamma 0}.$ However, by Condition (\ref{condB}), 
$\sum_{\gamma}T_{\gamma}\omega_{\gamma 0}(X) = \sum_iT_i\omega_{i 0}(X)+T_{n+1}\omega_{n+1,0}(X)
= \sum_i \frac{T_i}{ac}\big(\ep_i\produ{e_i,X}-\ep\ep_i T_i\eta(X)\big) -\ep\ep_{n+1}
\frac{T_{n+1}^2}{ac}\eta(X) = \frac{1}{ac}\big( \sum_i \ep_i\produ{e_i,T}\produ{e_i,X}-\ep \big(\sum_i\ep_i
T_i^2\big)\eta(X)-\ep\ep_{n+1} T_{n+1}^2\eta(X)\big)$,  and hence
\begin{equation}\label{(TOmega)n+2}
\sum_{\gamma}T_{\gamma}\omega_{\gamma 0}(X) =0.
\end{equation}
The case $\alpha=n+1$ and $\beta=0$ trivially vanishes due to $\omega_{n+1}=\omega_{0}=0$.
\end{proof}

\begin{lemma} \label{bigsum} $d\Upsilon+\Upsilon\wedge\Upsilon=0$.
\end{lemma}
\begin{proof}
This is equivalent to prove $d\X-d\Omega-\Omega\wedge \Omega
-X\wedge X+X\wedge \Omega+\Omega\wedge X=0.$ The case $\alpha=u$, $\beta=0$ is trivial. For the case $\alpha=i$ and $\beta=j$, we have
\begin{align*}
&(d\X-d\Omega-\Omega\wedge \Omega
-X\wedge X+X\wedge \Omega+\Omega\wedge X)_{ij}
\\
&=\frac{a''a-2(a')^2}{a^2}\eta\wedge(T_{j}\omega_{i}-\ep_i\varepsilon_{j}T_{i}\omega_{j})
+\frac{\varepsilon a'}{a}\Big(T_jd\omega_i -\ep_i\ep_jT_id\omega_j\Big)
\\
&\quad+
2\frac{\ep\ep_j (a')^2}{a^2}\omega_j\wedge\omega_i
+\frac{\varepsilon a'}{a}\Bigg(\sum_{u}
T_u(\omega_{uj}\wedge\omega_i-\ep_i\ep_j\omega_{ui}\wedge\omega_j)\Bigg)
\\
&\quad
+\ep \frac{(a')^2}{a^2} \ep_j\omega_i\wedge\omega_j
-\left(\frac{(a')^2}{a^2}-\frac{a''}{a}\right)\Big(T_j\omega_i-\ep_i\ep_jT_i\omega_j)\wedge\eta
\\
&\quad
-\left(\frac{a'}{a}\right)^2\Big(\big(T_{j}\omega_{i}
-\varepsilon_{i}\varepsilon_{j}T_{i}\omega_{j}
\big)\wedge\eta
-\ep\varepsilon_{j}\omega_{i} \wedge\omega_{j}\Big)
\\
&\quad
+\frac{\varepsilon a'}{a}\Big(\ep_{i}\ep_{j }T_id\omega_{j } -T_jd\omega_{i }
+\sum_{u}T_{u}\big(\omega_{i }\wedge\omega_{u j}
-\varepsilon_{j }\varepsilon_u\omega_{i u}\wedge\omega_{j}\big)\Big)
=0.
\end{align*}
Next, for $\alpha=i$, $\beta=n+1$, we compute
\begin{align*}
&(d\X-d\Omega-\Omega\wedge \Omega
-X\wedge X+X\wedge \Omega+\Omega\wedge X)_{i\,n+1}\\
&= \frac{a''a-2(a')^2}{a^2}T_{n+1}\eta  \wedge\omega_{i}
+\frac{\varepsilon a'}{a}\Big(\sum_kT_k\omega_{k\,n+1}\wedge \omega_{i} +T_{n+1}d\omega_{i}\Big) \\
&\quad -T_{n+1} \left(\frac{a''}{a} -\frac{(a')^2}{a^2}\right)\eta\wedge\omega_i- \left(\frac{a'}{a}\right)^2
\Big(
\big(T_{n+1}\omega_{i}
-\varepsilon_{i}\varepsilon_{n+1}T_{i}\omega_{n+1}
\big)\wedge\eta
-\ep\varepsilon_{n+1}\omega_{i} \wedge\omega_{n+1}
\Big) \\
&\quad
+\frac{\ep a'}{a}\Big(-T_{n+1}d\omega_i+\sum_{k}T_k\omega_i\wedge\omega_{k\,n+1}\Big)
\\
&=
\frac{a''a-2(a')^2}{a^2}T_{n+1}\eta  \wedge\omega_{i}
 -T_{n+1} \left(\frac{a''}{a} -\frac{(a')^2}{a^2}\right)\eta\wedge\omega_i  
- \left(\frac{a'}{a}\right)^2\big(T_{n+1}\omega_{i}\big)\wedge\eta =0.\qedhere
 \end{align*}
\end{proof}

It is clear that the map 
\[ s:\S\rightarrow \mathbb{S}(\mathbb{E}^{n+2})=\{X\in\mathbb{E}^{n+2}|\langle X,X\rangle\}=
\varepsilon_{n+1}\}, \quad
Z\mapsto (Z_{n+1,0},\ldots,Z_{n+1\,n+1})^t,
\]
is a submersion. Given a point $x\in M$, we define the set
\[\mathcal{Z}(x)=\{Z\in\S|Z_{n+1\beta}=T_{\beta}(x), \ \beta=0,\ldots,n+1\}.\]
Now, we prove the following
\begin{lemma} \label{prop_map1} Let $(M,\langle,\rangle)$ be a semi-Riemannian manifold satisfying
the structure conditions. For each $x_0\in M$ and $B_0\in \mathcal{Z}(x_0)$, there exists a
neighborhood $\mathcal{U}$ of $x_0$ in $M$ and a unique map $B: \mathcal{U}\rightarrow \S$, such
that
\begin{eqnarray*}
B^{-1}dB=\Omega-\mathbf{X}, \quad \textrm{ for all } x\in \mathcal{U},\quad B(x)\in \mathcal
Z(x),\quad
B_0=B(x_0).
\end{eqnarray*}
\end{lemma}
\begin{proof}
Given  $\mathcal{U}$ be an open neighborhood of $x_0\in M$, we define the set
$$\mathcal{F}=\{(x,Z)\in \mathcal{U}\times\S|Z\in \mathcal{Z}(x)\}.$$
Since the map $s$ is submersion, $\mathcal F$ is a submanifold of $M\times \S$ with
$$\dim \mathcal{F}=n+\frac{(n+1)(n+2)}{2}-(n+1)=\frac{n(n+1)}{2}+n.$$
Moreover, given $(x,Z)\in\mathcal{F}$,
$$T_{(x,Z)}\mathcal{F}=\{(U,V)\in T_x\mathcal{U}\oplus T_Z\S| V_{n+1\beta}=(dT_{\beta})_x(U), \
\beta=0,\ldots,n+1\}.$$
We consider on $\mathcal{F}$ the distribution $\mathfrak{D}(x,Z)=\ker\Theta_{(x,Z)}$, where
$\Theta=\Upsilon-Z^{-1}dZ=\Omega-\mathbf{X}-Z^{-1}dZ.$ 
In other words, given $(U,V)\in T_{(x,Z)}\mathcal{F}$, we have
$\Theta_{(x,Z)}(U,V)=\Omega_x(U)-\mathbf{X}_x(U)-Z^{-1}V.$
Next, we see that $\dim\mathfrak{D}=n$. We consider the space
$\mathcal{H}=\{H\in \mathfrak{s}|(ZH)_{n+1\beta}=0, \ \beta=0,\ldots,n+1\}.$
It is clear that $H\in\mathcal{H}$ if and only if $H\in\ker(ds)_{I_{n+2}}$. But the map $s$ is a
submersion, hence $\dim (\ker ds_{I_{n+2}})=\dim\mathcal{H}=\frac{(n+1)n}{2}$.
We notice that $(Z\Theta)_{n+1\beta}=(Z\Omega)_{n+1\beta}-(Z\X)_{n+1\beta}-(dZ)_{n+1\beta}=(Z\Omega)_{n+1\beta}
-(Z\X)_{n+1\beta} -dT_{\beta}.$ Consequently using equation \eqref{dT_k}  and $T_{0}=0$ we get
\begin{align*}
&(Z\Theta)_{n+1 k}=\sum_{\gamma}Z_{n+1\gamma}\omega_{\gamma k} -\sum_{\gamma}Z_{n+1\gamma}\X_{\gamma k} -\frac{a'}{a}\ep_k\omega_k +\ep\frac{a'}{a}T_k\eta
-\sum_{u}T_u\omega_{uk}\\
&=\sum_{\gamma}T_{\gamma}\omega_{\gamma k} -\sum_{\gamma}T_{\gamma}\frac{\ep
a'}{a}\big(T_{k}\omega_{\gamma}-\ep_{\gamma}\ep_{k}T_{\gamma}\omega_{k}\big)
-\frac{a'}{a}\ep_k\omega_k +\ep\frac{a'}{a}T_k\eta
-\sum_{u}T_u\omega_{uk}\\
&= -\frac{\ep
a'}{a}\big(T_{k}\eta-\sum_{\gamma}\ep_{\gamma}\ep_{k}T_{\gamma}T_{\gamma}\omega_{k}\big)
-\frac{a'}{a}\ep_k\omega_k +\ep\frac{a'}{a}T_k\eta
=\frac{\ep
a'}{a}\sum_{\gamma}\ep_{\gamma}\ep_{k}T_{\gamma}T_{\gamma}\omega_{k}-\frac{a'}{a}\ep_k\omega_k=0,
\end{align*}
since $\varepsilon=\langle T,T\rangle+\varepsilon_{n+1} T^2_{n+1}$. Similarly, 
\begin{align*}
&(Z\Theta)_{n+1\, n+1}=
\sum_{\gamma}T_{\gamma}\omega_{\gamma\,n+1}-\sum_{\gamma}T_{\gamma}\frac{\ep
a'}{a}\big(T_{n+1}\omega_{\gamma}-\ep_{\gamma}\ep_{n+1}T_{\gamma}\omega_{n+1}\big) -dT_{n+1}\\
&=\sum_{\gamma}T_{\gamma}\omega_{\gamma\,n+1}
-\frac{\ep a'}{a}T_{n+1}\eta -\sum_kT_k\omega_{k\,n+1}+\ep T_{n+1}\frac{a'}{a}\eta=0,
\end{align*}
and with equation \eqref{(TOmega)n+2}, it is clear 
\begin{align*}
(Z\Theta)_{n+1,0}&=
\sum_{\gamma}T_{\gamma}\omega_{\gamma 0}-\sum_{\gamma}T_{\gamma}\frac{\ep
a'}{a}\big(T_{0}\omega_{\gamma}-\ep_{\gamma}\ep_{0}T_{\gamma}\omega_{0}\big) -dT_{0}=0.
\end{align*}
Hence,
$\textrm{Im}(\Theta)\subset \mathcal{H}=\ker(ds)_{I_{n+2}}.$ Now, given the space
$\{(0,ZH)|H\in\mathcal{H}\}\subset T_{(x,Z)}\mathcal{F}$, we have
$\Theta_{(x,Z)}(0,ZH)=-Z^{-1}(ZH)=-H,$ which means that $\Theta_{(x,Z)}$ is a submersion unto $\mathcal{H}$, and  $\textrm{Im}(\Theta)= \mathcal{H}$. Now, we get $\dim \mathfrak{D}(x,Z)=\dim\ker\Theta_{(x,Z)}=\dim T_{(x,Z)}\mathcal{F}-\dim \textrm{Im}\Theta_{(x,Z)}=n.$

Next, we prove that $\mathfrak{D}$ is integrable. On one hand, since $\mathfrak{D}=\ker\Theta$ and
$d\Upsilon+\Upsilon\wedge \Upsilon=0$, we compute
$d\Theta =d\Upsilon+Z^{-1}dZ\wedge Z^{-1}dZ = d\Upsilon+(\Upsilon-\Theta)\wedge(\Upsilon-\Theta)
=-\Upsilon\wedge \Theta-\Theta\wedge\Upsilon+\Theta\wedge\Theta$. Therefore, given
$U,V\in\mathfrak{D}$, $d\Theta(U,V)=U(\Theta(V))-V(\Theta(U))-\Theta([U,V]) = -\Theta([U,V]) =
(-\Upsilon\wedge \Theta-\Theta\wedge\Upsilon+\Theta\wedge\Theta)(U,V) = 0$, which implies
$[U,V]\in\mathfrak{D}$.

Next, let $\mathcal{L}\subset\mathcal{F}$ be an integral manifold through $(x_0,B_0)$.  For each
$(0,V)\in\mathfrak{D}_{(x_0,B_0)}=T_{B_0}\mathcal{L}$, we have $\Theta_{(x_0,B_0)}(V)=B_0^{-1}V=0$,
and hence $V=0$, since $B_0\in \mathbf{S}$. This implies
$\mathfrak{D}_{(x_0,B_0)}\cap [\{0\}\times T_{B_0}\S]=\{0\}.$ 
In particular, by shrinking $\mathcal{U}$ if necessary, $\mathcal{L}$ is the graph of a unique map
$B:\mathcal{U}\rightarrow \mathbf{S}$. Also, since $\mathcal{L}\subset\mathcal{F}$, then for each
$x\in \mathcal{U}$, $B(x)\in\mathcal{Z}(x)$.
Finally since $\Theta\equiv 0$ on $\mathcal{L}$, $B$ satisfies by definition
$B^{-1}dB=\Omega-\mathbf{X}$.
\end{proof}

Define now the map $\chi:\mathcal{U}\rightarrow\mathbb{R}^{n+2}$ by
\[ \chi_{0}=\varepsilon_{0}B_{0 0},\quad \chi_{i}=\varepsilon_iB_{i0},\quad  \chi_{n+1}=\pi.
\]
Notice that, since $B(x)\in \mathcal{Z}(x)\subset\mathbf{S}$, then $B_{n+1, 0}=T_{0}=0$, whic implies 
$\ep_0 \chi_{0}^2+\sum_{i=1}^{n}\ep_i\chi_i^2=\sum_{\alpha=0}^{n} \ep_{\alpha} B_{\alpha
0}^2=\ep_0=c,$ thus obtaining that $(\chi_0,\dots,\chi_{n})$ lies in $\mathbb{M}^n_k(c)$, which means
$\textrm{Im}(\chi)\subset\varepsilon I\times_{a}\mathbb{M}^n_k(c)$. 
Now, we have by definition $dB=B\Omega-B\X$. Hence
\begin{align*}
d\chi_i(e_k)&=\varepsilon_idB_{i0}(e_k) =\sum_{\alpha=0}^{n+1}\varepsilon_i(B_{i\alpha}\omega_{\alpha 0}(e_k)-B_{i\alpha}\X_{\alpha 0}(e_k))\\
&=\varepsilon_i\sum_{\alpha=1}^{n+1}B_{i\alpha}\omega_{\alpha
0}(e_k)-\varepsilon_i\sum_{\alpha=1}^{n+1}B_{i\alpha}\Big(\frac{\varepsilon
a'}{a}(T_{0}\omega_{\alpha}(e_k)-\varepsilon_{0}\varepsilon_{\alpha}T_{\alpha}\omega_{0}(e_k)\Big)\\
&=\varepsilon_i\sum_{j=1}^nB_{ij}\omega_{j 0}(e_k)+B_{in+1}\omega_{n+1 0}(e_k)\\
&=\varepsilon_i\sum_{j=1}^nB_{ij}\frac{1}{ac}(\varepsilon_j\langle
e_j,e_k\rangle-\varepsilon\varepsilon_jT_j\eta(e_k))-
B_{in+1}\varepsilon\varepsilon_{n+1}\frac{1}{ac}T_{n+1}\eta(e_k)\\
&=\varepsilon_i\frac{1}{ac}B_{ik}-\frac{\varepsilon\varepsilon_{n+1}}{ac}B_{in+1}T_{n+1}T_k=\frac{
\varepsilon_i}{ac}B_{ik}
\end{align*}
A similar computation yields $d\chi_{0}(e_k)=\varepsilon_{0}\frac{1}{ac}B_{0k}$
and $d\chi_{n+1}(e_k)=\varepsilon\eta(e_k)=\varepsilon T_k=\varepsilon B_{n+1k}$.  Hence, we have
that $d\chi=CB\varpi$,
with 
\[C= \begin{pmatrix}
  \varepsilon_0/ca & 0 & \cdots &0 \\
0  & \ddots  & \ddots & \vdots  \\
\vdots & \ddots &  \varepsilon_{n}/ca & 0\\
  0 & \cdots &  0 & \varepsilon_{n+1}
 \end{pmatrix}, \quad \varpi=(0,\omega_1,\cdots\omega_n,0)^T,\]
or equivalently, $d\chi(e_k)_{\alpha}=(CB)_{\alpha k}$, meaning that in the frame
$\frac{\partial}{\partial x_{\alpha}}$ the vector $d\chi(e_k)$ is given by the k-th column of the
matrix $CB$ and in the frame $\bar{E}_{\alpha}$ by the k-th column of the matrix $B$. In other
words, $d\chi (e_k)=\sum_{\alpha}\varepsilon_{\alpha}B_{\alpha k}\bar{E}_{\alpha}$. $C$ is an
invertible matrix as well as $B$. Consequently, $d\chi$ has rank n and it is an immersion.
Moreover, for any $i,j$,  since $B\in\mathbf{S}$, we have
\[
\langle d\chi(e_i),d\chi(e_j)\rangle=
\produ{ \sum_{\alpha}\varepsilon_{\alpha}B_{\alpha i}\bar{E}_{\alpha} ,
\sum_{\gamma}\varepsilon_{\gamma}B_{\gamma j}\bar{E}_{\gamma} } 
=  \sum_{\gamma} \ep_{\alpha} B_{\alpha i}B_{\alpha j} = \ep_i \delta_{ij},
\]
Hence, $\chi$ is isometric. Moreover, along $\chi$, we obtain
\[
\frac{\partial}{\partial t}=\sum_{i=1}^n\ep_i\langle \frac{\partial}{\partial t},
d\chi(e_i)\rangle d\chi(e_i)
+\ep_{n+1}\langle\frac{\partial}{\partial t},e_{n+1}\rangle e_{n+1}
=T+\varepsilon_{n+1}T_{n+1}e_{n+1} .
\]
Next, we would like to compute the shape operator of the immersion. Recall
$\bar{E}_{n+1}=\partial_t$. We show that the shape operator of the immersion is exactly what we need, namely $d\chi\circ A\circ d\chi^{-1}$. Indeed, 
\begin{align*}
&\langle \bar{\nabla}_{d\chi_{(e_i)}}d\chi_{(e_j)},e_{n+1}\rangle 
=\langle \sum_{\alpha\beta}\bar{\nabla}_{\varepsilon_{\alpha}B_{\alpha
i}\bar{E}_{\alpha}}\varepsilon_{\beta}B_{\beta j}\bar{E}_{\beta},e_{n+1}\rangle=\langle
\sum_{\alpha\beta}\varepsilon_{\alpha}\varepsilon_{\beta}\widetilde{\nabla}_{B_{\alpha
i}\bar{E}_{\alpha}}B_{\beta j}\bar{E}_{\beta},e_{n+1}\rangle
\\
&= \sum_{\alpha\beta\gamma}\big[\varepsilon_{\alpha}\varepsilon_{\beta}\varepsilon_{\gamma}B_{\alpha
i}B_{\beta
j}B_{\gamma\,n+1}\langle\widetilde{\nabla}_{\bar{E}_{\alpha}}\bar{E}_{\beta},\bar{E}_{\gamma}\rangle
+
\varepsilon_{\beta}dB_{\beta j}(e_i)B_{\beta n+1}\big]
+\sum_{\beta}\varepsilon_{\beta}dB_{\beta j}(e_i)B_{\beta n+1}
\\
&=\sum_{uv\gamma}\big[\varepsilon_u\varepsilon_v\varepsilon_{\gamma}B_{u i}B_{v j}
B_{\gamma\,n+1}\langle-\frac{\varepsilon_u\delta_{uv}\varepsilon
a'}{a}\partial_t,\bar{E}_{\gamma}\rangle +\varepsilon_u\varepsilon_{n+1}\varepsilon_{\gamma}B_{u
i}B_{n+1 j}B_{\gamma\,n+1}\langle\frac{a'}{a}\bar{E}_u,\bar{E}_{\gamma}\rangle\big]
\\
&\quad 
+\sum_{\beta}\varepsilon_{\beta}dB_{\beta j}(e_i)B_{\beta n+1}
\\
&=-\frac{ \varepsilon a'}{a}B_{n+1 n+1}\sum_{u}\varepsilon_u B_{u i}B_{u
j}+\frac{\varepsilon_{n+1}a'}{a}B_{n+1 j} \sum_u\varepsilon_uB_{u i}B_{u n+1}
+\varepsilon_{n+1}(B^{-1}dB)_{n+1 j}(e_i)
\\&=
-\frac{\varepsilon a'}{a}B_{n+1 n+1}([\varepsilon_i\delta_{ij}-\varepsilon_{n+1}B_{n+1i}B_{n+1 j}]
\\
&\quad +\frac{\varepsilon_{n+1}a'}{a}B_{n+1 j}[\varepsilon_i\delta_{i
n+1}-\varepsilon_{n+1}B_{n+1i}B_{n+1n+1}]
+\varepsilon_{n+1}(B^{-1}dB)_{n+1 j}(e_i)\\
&=\varepsilon_{n+1}\left[(B^{-1}dB)_{n+1 j}(d\chi_{(e_i)})-\frac{\varepsilon
a}{a'}[\varepsilon_{n+1}\varepsilon_jT_{n+1}\omega_j(e_i)-T_j\omega_{n+1}(e_i)]\right]\\
&= \varepsilon_{n+1}[(B^{-1}dB)_{n+1j}+X_{n+1 j}]=\varepsilon_{n+1}\omega_{n+1 j}(e_i)
=\langle e_j, Ae_i\rangle.
\end{align*}
Finally, the uniqueness of the local immersion follows from the uniqueness of the map $B$ in Lemma
\ref{prop_map1}.
\end{proof}

\begin{corollary}\begin{enumerate}
\item If the hypersurface $M$ satisfies $\eta=0$, then, $M$ is a slice of $\mmm$.
\item If $\eta\neq 0$ everywhere, then $M$ is admits a foliation of codimension 1.
\end{enumerate}
\end{corollary}
\begin{proof} Item 1 is an immediate consequence of item \ref{dt} of Theorem \ref{main}. For item 2,
by Lemma \ref{deta}, we know $d\eta =0$.  This implies that $\ker\eta$ is integrable. Indeed, given
$X,Y\in\ker\eta$, $0=d\eta(X,Y)=X(\eta(Y))-Y(\eta(X))-\eta([X,Y])=-\eta([X,Y])$, which shows
$[X,Y]\in\ker\eta$. In other words, it has to admit a foliation whose leaves are of codimension 1 in
$M$. In fact, $T$ is a normal vector field to the leaves. 
\end{proof}

\noindent \textbf{Remark:} Under the same assumption on $(M,\langle,\rangle,\nabla, R)$,  we can find another equivalent formulation for Theorem \ref{main}. In fact, consider again a $\langle ,\rangle$-self adjoint $(1,1)$-tensor $A$ on $M$, a nowhere vanishing vector field $T\in\mathfrak{X}(M)$ and its associated 1-form $\eta(X)=\produ{X,T}$ for any $X\in TM$. We also assume the existence of smooth functions $\rho,
\tilde\rho, \bar\rho, T_{n+1}:M\rightarrow \R$. Let the following conditions be satisfied:
\begin{enumerate}[(a)]
\item $\rho>0$, $d\rho = \ep \tilde\rho\, \eta$, $d\tilde\rho =\ep \bar\rho\, \eta$;
\item $\ep = \produ{T,T}+\ep_{n+1}T_{n+1}^2$;
\item $\nabla_XT=\frac{\tilde\rho}{\rho}(X-\ep\eta(X)T)+\ep_{n+1}T_{n+1}AX$, for any $X\in TM$;
\item  $X(T_{n+1})=-\produ{AT,X}-\ep\frac{\tilde\rho}{\rho} T_{n+1}\eta(X)$, for any $X\in TM$;
\item (Codazzi equation). For any $X,Y,Z,W\in TM$ \label{eq-codazziMx}
\[(\nabla_XA) Y - (\nabla_YA) X =
T_{n+1}  \left(\frac{\bar\rho}{\rho}
-\left(\frac{\tilde\rho}{\rho}\right)^2+\frac{\ep\ep_{0}}{\rho^2}\right) \Big(\eta(Y)X
-\eta(X)Y\Big).\]
\item (Gau\ss\, equation). \label{eq-GaussMx} For any $X,Y,Z,W\in TM$
\begin{align*}
&R(X,Y,Z,W)=
 \left( \ep\left(\frac{\tilde\rho}{\rho}\right)^2-\frac{\ep_{0}}{\rho^2}  \right)
 \Big( \produ{X,Z}\produ{Y,W}-\produ{Y,Z}\produ{X,W}\Big) \nonumber \\
 &\quad  +\left(\frac{\bar\rho}{\rho}
-\left(\frac{\tilde\rho}{\rho}\right)^2+\frac{\ep\ep_{0}}{\rho^2}\right)
 \Big(\produ{X,Z}\eta(Y)\eta(W) -\produ{Y,Z}\eta(X)\eta(W)
\\ &\quad
- \produ{X,W}\eta(Y)\eta(Z)+ \produ{Y,W}\eta(X)\eta(Z)\Big)
 +\ep_{n+1}\Big( \produ{AY,Z}\produ{AX,W}-\produ{AY,W}\produ{AX,Z}\Big).
\end{align*}

\end{enumerate}
Then Theorem \ref{main} can be reformulated in the following way:
\begin{theorem}  Let $(M,\produ{,})$ a simply connected semi-Riemannian manifold satisfying the
previous conditions. Then, there exists smooth functions $\pi:M\rightarrow I$, $I$ an interval,
$a:I\subset\R\rightarrow\R^{+}$,   a metric immersion $\chi:(M,\produ{,})\rightarrow
(\mmm,\produ{,}_1)$ and a normal unit vector field $e_{n+1}$ along $\chi$ such that:
\begin{enumerate}
\item $\ep_{n+1}=\produ{e_{n+1},e_{n+1}}_1$;
\item $\pi_I\circ\chi=\pi$, where $\pi_I:\mmm\rightarrow I$ is the projection;
\item $\rho = a\circ \pi$, $\tilde\rho = a'\circ \pi$ and $\bar\rho = a''\circ \pi$; 
\item The shape operator associated with $e_{n+1}$ is $A$;
\item (\ref{eq-codazziMx}) is the Codazzi equation and (\ref{eq-GaussMx}) is the Gau\ss\, equation;
\item   and along $\chi$, 
$\dt = T+\ep_{n+1}T_{n+1}e_{n+1}$ holds.
\end{enumerate}
\end{theorem} 

\begin{proof} First of all, from the expression $\nabla_XT =
\frac{\tilde\rho}{\rho}(X-\ep\eta(X)T)+\ep_{n+1}T_{n+1}AX$, for any $X\in TM$,
we are going to check that  the 1-form $\eta$ satisfies $d\eta =0$. 
\begin{align*}
&d\eta(e_i,e_j) = e_i(\eta(e_j))-e_j(\eta(e_i))-\eta(\nabla_{e_i}e_j)+\eta(\nabla_{e_j}e_i) =
\produ{e_j,\nabla_{e_i}T}-\produ{e_i,\nabla_{e_j}T} \\
&=
  \produ{e_j, \frac{\tilde\rho}{\rho}(e_i-\ep\eta(e_i)T)+\ep_{n+1}T_{n+1}Ae_i}
- \produ{e_i,\frac{\tilde\rho}{\rho}(e_j-\ep\eta(e_j)T)+\ep_{n+1}T_{n+1}Ae_j} \\
&=
\frac{\tilde\rho}{\rho}\big(
\produ{e_j,e_i}-\ep\eta(e_i)\eta(e_j)\big)+\ep_{n+1}T_{n+1}\produ{e_j,Ae_i}
-\frac{\tilde\rho}{\rho}\big( \produ{e_i,e_j}-\ep\eta(e_i)\eta(e_j)\big)\\
&\quad -\ep_{n+1}T_{n+1}\produ{e_i,Ae_j} =0.
\end{align*}
Since $M$ is simply connected, we can obtain a new function $\pi:M\rightarrow\R$ such that $\eta =
\ep d\pi$.
This implies that $T=\ep\grad(\pi)$.  Next, we need to obtain function $a$. On one hand, since $M$
is connected, $I=\pi(M)$ is an interval. Moreover, since $T\neq 0$, we see that each value $t\in I$
is a regular value of $\pi$, which means that each level set $\pi^{-1}(t)\subset M$, $t\in I$, is a
hypersurface of $M$. Choose $t\in I$. Given $X\in T\pi^{-1}(t)$, since $\pi$ is constant along its
level subsets, we see $d\rho(X)=\ep\tilde\rho\,\eta(X) = \tilde{\rho}\,d\pi(X) =0$. In other words,
function $\rho$ is constant along the level sets of $\pi$. This allows us to define
$a:I\rightarrow\R^{+}$ as follows. Given $t\in I$, there exists $p\in M$ such that $t=\pi(p)$, so
that $a(t):=\rho(p)$. Clearly, $\rho = a\circ \pi$. In addition, $d\rho = (a'\circ\pi)d\pi =(a'\circ
\pi)\ep\eta = \ep \tilde{\rho}\eta$, and therefore $\tilde{\rho}=a'\circ\pi$. Similarly,
$a''\circ\pi = \bar\rho$. Now, we just need to resort to Theorem \ref{main}.
\end{proof}

\section{An Application to Horizons in RW 4-spacetimes}

We consider now the simply connected Riemannian 3-dimensional space $\mathbb{M}^3(c)$ of constant
sectional curvature $c=\pm 1$. Let $(M^3,\produ{,})$ be a semi-Riemannian manifold of index $0$ or
$1$. For us, a surface $\tilde{M}^2$ is called \textit{marginally trapped} if its mean curvature
vector $\vec{H}$ satisfies $\produ{\vec{H},\vec{H}}=0$. In this way, we are including maximal
surfaces,   MOTS, and mixed cases in our definition.

We put $\ep=-1$, $\ep_0=c$, $\ep_{4}=\pm 1$, and smooth functions $a:I\subset\R\rightarrow\R^{+}$,
$T_{4}:M\rightarrow\R$ and $\pi:M\rightarrow I$. We construct the vector field $T\in\mathfrak{X}(M)$
by $T=- \grad(\pi)$,  with its 1-form $\eta(X)=\produ{X,T}$. Also, consider a  tensor $A$ of type
(1,1) on $M$. We assume the above datas satisfy the structure conditions of Definition \ref{structureconditions} . 
We recall that the Robertson-Walker space-time is the space$(\bar{P}^{4}=I\times
\mathbb{M}^3(c),\produ{,}_1=-dt^2+a^2g_o)$, hence a special case of the warped products considered in this paper. From Theorem \ref{main}, we get immediately the following

\begin{corollary} \label{RW} Let $(M,\produ{,})$ a semi-Riemannian manifold of $\dim M=3$,
satisfying the previous conditions. Then, for each point $p\in M$, there exists a neighborhood
$\mathcal U$ of p on $M$, a metric immersion $\chi:(\mathcal {U},\produ{,})\rightarrow
(\bar{P}^{4},\produ{,}_1)$ and a normal unit vector field $e_{4}$ along $\chi$ such that $\ep_{4}=\produ{e_{4},e_{4}}_1$, $A$ is the shape operator associated to the immersion, $T$ is the projection of $\partial_t$ on $TM$ and $\pi_I\circ\chi=\pi$, where $\pi_I:I\times\mathbb{M}^3(c)\rightarrow I$ is the projection.

In addition, if $T\neq 0$ everywhere,  the family $\{\chi(\mathcal{U}) \cap \pi^{-1}\{t\} : t\in \R
\}$ provides a foliation of $\chi(\mathcal{U})$ by space-like surfaces.
\end{corollary}

Next, let $L$ be one of the leafs of $\mathcal{U}$. Let $\sigma$ be its second fundamental form in
$\bar{P}^4$. Clearly, $T^{\perp}L=\mathrm{Span}\{\mathbf{T},e_4\}$, where $\mathbf{T}=T/\sqrt{\vert
\produ{T,T}\vert}$. We take $\ep_{T}=sign(\produ{T,T})$. Since the leaves are spacelike and
$\produ{e_4,e_4}=\ep_4=\pm 1$ is constant, $\ep_{T}=\pm 1$ is constant,
with $\ep_4\ep_T=-1$. Then, for any $X,Y\in TL$,
\[ \sigma(X,Y) = \ep_T \produ{\nat_XY,\mathbf{T}}\mathbf{T} +\ep_4 \produ{\nat_XY,e_4}e_4
=\frac{-1}{\produ{T,T}}\produ{Y,\nabla_XT}T +\ep_4\produ{Y,AX}e_4,
\]
\nolinebreak Given a local orthonormal frame $\{u_1,u_2\}$ of $L$, the mean curvature vector $\vec H$ of $L$ is 
\begin{align*}
2\vec{H }&=  \sum_{i} \sigma(u_i,u_i) = 
\sum_i \Big(\frac{-1}{\produ{T,T}}\produ{u_i,\nabla_{u_i}T}T +\ep_4\produ{u_i,Au_i}e_4 \Big)\\
&= \frac{-1}{\produ{T,T}}\sum_i\produ{u_i,\frac{a'}{a}u_i+\ep_4T_4Au_i} T+\ep_4\Big(
\mathrm{trace}(A) - \frac{\produ{AT,T}}{\produ{T,T}}\Big) e_4 \\
&= \frac{-1}{\produ{T,T}}\Big( 2\frac{a'}{a} +\ep_4T_4 \Big(\mathrm{trace}(A)  -
\frac{\produ{AT,T}}{\produ{T,T}}\Big)\Big) T+\ep_4\Big( \mathrm{trace}(A) -
\frac{\produ{AT,T}}{\produ{T,T}}\Big) e_4.
\end{align*}
We put $h=\mathrm{trace}(A)  - \frac{\produ{AT,T}}{\produ{T,T}}$. As a result, we obtain
$
4\produ{\vec{H},\vec{H}}
=  \frac{1}{\produ{T,T}}\Big( 2\frac{a'}{a} +\ep_4T_4 h\Big)^2+\ep_4 h^2.
$
Then $\vec{H}$ is null if, and only if, $-\ep_4 h^2= \frac{1}{\produ{T,T}}\Big( 2\frac{a'}{a} +\ep_4T_4 h\Big)^2 =  \frac{1}{\produ{T,T}}\Big( 4\frac{(a')^2}{a^2} +T_4^2 h^2 + \frac{4\ep_4 a' T_4 h}{a}
 \Big)$, which is equivalent to $h^2-\frac{4a'T_4}{a} h-\frac{4\ep_4(a')^2}{a^2}=0$. Since $\ep_T\ep_4=-1$, we see that $-\ep_4\produ{T,T}>0$, and so, by solving this equation, we
obtain the following
\begin{corollary}\label{horizon} The leaves of $\ker\eta$ are marginally trapped surfaces in
$-I\times_a\mathbb{M}^3(c)$ if, and only if, the immersion $\chi: M\rightarrow
-I\times_a\mathbb{M}^3(c)$ satisfies the following equality:
\[\mathrm{trace}(A)-\frac{\produ{AT,T}}{\produ{T,T}}=2a'T_4 \pm 2\ep_T \left\vert \frac{a'}{a}
\right\vert \sqrt{\vert\produ{T,T}\vert }.\]
\end{corollary}

\section{Examples}

\begin{example}{A graph surface in $-\R^{+}\times_a \mathbb{S}^2$. }
We start by considering the warping function $a:\R^{+}\rightarrow\R^{+}$, $a(t)=t$, and the warped
product $-\R^{+}\times_a\mathbb{S}^2$, with metric $\produ{,}$. Given $h:(0,\pi)\rightarrow \R^{+}$ a smooth function such that $h(u) > \vert h'(u)\vert$ for any
$u\in \R^{+}$, we introduce the map
\begin{align*}
& \chi :M=(0,\pi)\times (-\pi/2,\pi/2)\rightarrow -I\times_a\mathbb{S}^2, \\
& \chi (u,v) = (\cos(u),\sin(u)\cos(v), \sin(u) \sin(v), h(u)),
\end{align*}
Note that $\pi(u,v)=h(u)$. 
We consider the following frame along $\chi $, with some natural identifications: 
\begin{align*}
e_0 = & \frac{1}{h(u)}\big(\cos(u), \cos(v) \sin(u),\sin(u) \sin(v), 0\big),  \\
e_1 \equiv& \chi_{*}e_1= \frac{1}{ \sqrt{ h(u)^2 - (h'(u))^2} } \big(-\sin(u), \cos(u)\cos(v),\cos(u)\sin(v),
h'(u)\big), \\
e_2 \equiv&\chi_*e_2 =  \frac{1}{h(u)} \big(0, -\sin(v), \cos(v), 0\big), \\
e_3 = & \frac{1}{ h(u) \sqrt{ h(u)^2 - (h'(u))^2} } \big(-\sin(u) h'(u) ,\cos(u)\cos(v) h'(u),
\cos(u)\sin(v) h'(u),  h(u) \big), 
\end{align*}
where $e_1,e_2$ are the normalizations of $\chi _{*}\partial_u$ and $\chi _{*}\partial_v$,
respectively, and $e_3$ is a unit vector field of $M$ along $\chi $ on
$-\R^{+}\times_a\mathbb{S}^2$. Also, the matrix $G=(\produ{e_{\alpha},e_{\beta}})=Diag(1,1,1,-1)$. 
The dual 1-forms on $M$ are 
\[ \omega_0 = 0, \quad \omega_1 = \sqrt{ h(u)^2-(h'(u))^2 } \mathrm{d}u, \quad \omega_2 =
h(u)\sin(u)\mathrm{d}v, \quad \omega_3 = 0.
\]
Since $\R^{+}\times \mathbb{S}^2\subset \R^{+}\times\R^3$, we consider the orthonormal basis 
\[ \E_0=\frac{1}{h}(1,0,0,0), \quad \E_1=\frac{1}{h}(0,1,0,0),\quad \E_2=\frac{1}{h}(0,0,1,0),\quad
\E_3=\partial_t.
\]
Now, we can compute the map $B:M\rightarrow\S$, $B = (\produ{\E_{\alpha},e_{\beta}})$, 
\[ B=\left(
\begin{matrix}
\vspace{2mm}\cos(u) & \frac{-h(u)\sin(u)}{\sqrt{ h(u)^2-(h'(u))^2 } } & 0 & \frac{ - \sin(u) h'(u)
}{ \sqrt{ h(u)^2-(h'(u))^2 } } \\ \vspace{2mm}
-\cos(v)\sin(u) & \frac{-\cos(u)\cos(v) h(u)}{\sqrt{ h(u)^2-(h'(u))^2 }} & \sin (v) &
\frac{-\cos(u)\cos(v) h'(u)}{\sqrt{ h(u)^2-(h'(u))^2 }} \\
\vspace{2mm} -\sin(u)\sin(v) & \frac{-\cos(u) h(u)\sin(v) }{\sqrt{ h(u)^2-(h'(u))^2 }} & -\cos(v) &
\frac{-\cos(u) \sin(v) h'(u)}{\sqrt{ h(u)^2-(h'(u))^2 }} \\
0 & \frac{-h'(u)}{\sqrt{ h(u)^2-(h'(u))^2 }} & 0 & \frac{-h(u)}{\sqrt{ h(u)^2-(h'(u))^2 }}
\end{matrix}
\right).
\]
From this, a straightforward computation gives the $\mathfrak{s}$-valued 1-form $\Upsilon = B^{-1}\mathrm{d}B=(\Upsilon_{\alpha\beta})$, 
\begin{gather*}
\Upsilon_{\alpha\alpha}=0,\quad \Upsilon_{01}=-\Upsilon_{10}= \frac{-h(u)\mathrm{d}u}{\sqrt{ h(u)^2-(h'(u))^2 }}, \quad \Upsilon_{02}=-\Upsilon_{20} = -\sin(u)\mathrm{d}v, \\
\Upsilon_{03}=\Upsilon_{30} = \frac{-h'(u)\mathrm{d}u}{\sqrt{ h(u)^2-(h'(u))^2 }}, \quad 
\Upsilon_{12}=-\Upsilon_{21} = \frac{-h(u)\cos(u)\mathrm{d}v}{\sqrt{h(u)^2-(h'(u))^2 }}, \\
\Upsilon_{13}=\Upsilon_{31}= \frac{h(u)h''(u)-(h'(u))^2}{h(u)^2-(h'(u))^2}\mathrm{d}u, \quad 
\Upsilon_{23}=\Upsilon_{32}= \frac{h'(u)\cos(u)\mathrm{d}v}{\sqrt{ h(u)^2-(h'(u))^2 }}. 
\end{gather*}
A lenghty computation shows $\mathrm{d}\Upsilon+\Upsilon\wedge\Upsilon= 0$. Next, the tangent vector $T$ is 
\begin{align*} T&\equiv \chi _{*}T = \produ{\partial_t,e_1}e_1+\produ{\partial_t,e_2}e_2 \\
&=\Big(\frac{\sin(u)h'(u)}{ h(u)^2-(h'(u))^2 }, \frac{ -\cos(u)\cos(v) h'(u) }{ h(u)^2-(h'(u))^2 }, 
\frac{-\cos(u)\sin(v)h'(u) } { h(u)^2-(h'(u))^2 }, \frac{-(h'(u))^2} { h(u)^2-(h'(u))^2 }
\Big).
\end{align*}
Thus, its dual 1-form and the associated functions $T_{\alpha}$ are
\[\eta  =  -h'(u)\mathrm{d}u, \quad 
T_0=0, \quad T_1 = -\frac{h'(u)}{\sqrt{ h(u)^2-(h'(u))^2 } }, \quad T_2=0, \quad T_3=T_1.
 \]
 Clearly, $\mathrm{d}\eta=0$. Also, note that $T_{\alpha} = B_{3\alpha}$. 
 We recall that $\Omega = (\omega_{\alpha\beta} ) = B^{-1}\mathrm{d}B+\X$, where
 $\X=(\X_{\alpha\beta})$, 
$\X_{\alpha\beta}= \ep \frac{a'\circ \pi}{a\circ \pi}\Big(
B_{n+1\beta}\omega_{\alpha}-\ep_{\alpha}\ep_{\beta}B_{n+1\alpha}\omega_{\beta} \Big)$,
\begin{gather*} \X_{\alpha\alpha}= \X_{\alpha 0}=\X_{0\alpha} = 0, 
\quad
\X_{13}=\X_{31}=\mathrm{d}u,
\\
\X_{12} = -\X_{21} = \frac{-\sin(u)h'(u)\mathrm{d}v}{ \sqrt{ h(u)^2-(h'(u))^2 } },  \quad
\X_{23}=\X_{32}=\frac{\sin(u)h(u)\mathrm{d}v}{\sqrt{ h(u)^2-(h'(u))^2 } }.
\end{gather*}
In this way, we have
\begin{align*}
\Omega = (\omega_{\alpha\beta}), \quad
\omega_{01}=-\omega_{10}=\frac{-h(u)\mathrm{d}u}{\sqrt{ h(u)^2-(h'(u))^2 }}, \quad
\omega_{02}= -\omega_{20}= -\sin(u)\mathrm{d}v, \\
\omega_{03}=\omega_{30} =  \frac{-h'(u)\mathrm{d}u}{\sqrt{ h(u)^2-(h'(u))^2 }}, \quad
\omega_{12}=-\omega_{21} = -\frac{\cos(u)h(u)+\sin(u)h'(u)}{\sqrt{ h(u)^2-(h'(u))^2 }}\mathrm{d}v,
\\
\omega_{13}=\omega_{31} = \frac{h(u)^2-2(h'(u))^2+h(u)h''(u)}{h(u)^2-(h'(u))^2}\mathrm{d}u, \quad
\omega_{23}=\omega_{32} = \frac{\cos(u)h'(u)+\sin(u)h(u)}{\sqrt{ h(u)^2-(h'(u))^2 }}\mathrm{d}v. 
\end{align*}
Now, we compute the shape operator. Since $\omega_{i3}(X) = -\ep_i\produ{AX,e_i}=-\omega_i(AX)$,
then
\begin{align*}
AX &=\sum_i \omega_i(AX)e_i =-\sum_i\omega_{i3}(X)e_i \\
&=\frac{h(u)^2-2(h'(u))^2+h(u)h''(u)}{h(u)^2-(h'(u))^2}\mathrm{d}u(X)e_1 
+\frac{\cos(u)h'(u)+\sin(u)h''(u)}{\sqrt{ h(u)^2-(h'(u))^2 }} \mathrm{d}v(X)e_2,  \\
Ae_1 &= \frac{h(u)^2-2(h'(u))^2+h(u)h''(u)}{h(u)(h(u)^2-(h'(u))^2)}e_1, \quad 
Ae_2 = \frac{\cos(u)h'(u)+\sin(u)h''(u)}{ h(u)(h(u)^2-(h'(u))^2) } e_2.
\end{align*}
\end{example}

\begin{example}{A helicoidal surface in $\R\times_a\mathbb{H}^2$}
We consider now $\mathbb{H}^2(-1)$ as the surface $\mathbb{H}^2(-1) = \{ (x,y,z)\in\mathbb{L}^3 : x^2+y^2-z^2=-1\}$, where $\mathbb{L}^3$ is the Lorentz-Minkowski space endowed with the standard metric
$\produ{,}=dx^2+dy^2-dz^2$. Given a real constant $c\in\R$ and a smooth function $h:\R\rightarrow\R$ with $h'>0$, we construct 
\[
\chi :M=\R^{+}\times\R\longrightarrow \R\times_a\mathbb{H}^2, \quad 
\chi (u,v) = (u\cos(c v), u\sin(c v), \sqrt{1+u^ 2},h(v)).
\]
Note that $\pi(u,v)=h(v)$.  
We consider the following frame along $\chi $, with some natural identifications: 
\begin{align*}
e_0 = & \frac{1}{a(h(v))} \big(u\cos(c\, v), u \sin(c\, v) ,\sqrt{1+u^2}, 0\big),  \\
e_1 \equiv \chi_* e_1=& \frac{1}{ a(h(v))} \big(\cos(c\, v)\sqrt{1+u^2}, \sin(c\, v)\sqrt{1+u^2},u, 0\big), \\
e_2 \equiv \chi_* e_2= &  \frac{1}{\sqrt{c^2 u^2 a(h(v))^2+h'(v)^2} } 
\big(-c\, u\, \sin(c\, v), c\, u\, \cos(c\, v), 0, h'(v)\big), \\
e_3 = & \frac{1}{a(h(v)) \sqrt{c^2 u^2 a(h(v))^2+h'(v)^2} } 
\big(\sin(c\, v) h'(v) ,-\cos(c\, v) h'(v), 0, c\, u\, a(h(v)) \big), 
\end{align*}
where $e_1,e_2$ are the normalizations of $\partial_u$ and $\partial_v$, respectively, and $e_3$ is a unit vector field of $M$ along $\chi $ on $\R\times_a\mathbb{H}^2$. The dual 1-forms on $M$ are  
\[ \omega_0 = 0, \quad \omega_1 = \frac{a(h(v))}{\sqrt{1+u^2}}\mathrm{d}u, \quad 
\omega_2 = \sqrt{c^2 u^2 a(h(v))^2+h'(v)^2}\, \mathrm{d}v, \quad \omega_3 = 0.
\]
We need to introduce the basis
\[ \E_0=\frac{1}{a}(0,0,1,0), \quad \E_1=\frac{1}{a}(0,1,0,0),\quad \E_2=\frac{1}{a}(1,0,0,0),\quad
\E_3=\partial_t.
\]
Now, we can compute the map $B:M\rightarrow\S$, $B = (\produ{\E_{\alpha},e_{\beta}})$, with $\det B=1$, 
\[ B=\left(
\begin{matrix}
\vspace{2mm}
-\sqrt{1+u^2} & - u & 0 & 0 \\
\vspace{2mm} u\,\sin(c\,v) & \sqrt{1+u^2}\sin(c\,v) & \frac{c u a(h(v))\cos(cv)}{\sqrt{c^2 u^2 a(h(v))^2+h'(v)^2}} & 
-\frac{\cos(c v)h'(v)}{\sqrt{c^2 u^2 a(h(v))^2+h'(v)^2}} \\
\vspace{2mm} u\cos(cv) & \sqrt{1+u^2}\cos(c\,v) & -  \frac{c u a(h(v))\sin(cv) }{ \sqrt{c^2 u^2 a(h(v))^2+h'(v)^2} } &
\frac{ \sin(cv)h'(v) }{ \sqrt{c^2 u^2 a(h(v))^2+h'(v)^2} } \\
\vspace{2mm}0 & 0 & \frac{ h'(v) }{ \sqrt{c^2 u^2 a(h(v))^2+h'(v)^2 }} & \frac{c u a(h(v)) }{\sqrt{c^2 u^2 a(h(v))^2+h'(v)^2}} \\
\end{matrix}
\right).
\]
Note that $B$ does not lie in the orthogonal group $O(4)$. From this, a straightforward computation gives the $\mathfrak{s}$-valued 1-form 
$\Upsilon =(\Upsilon_{\alpha\beta})= B^{-1}\mathrm{d}B$, where
\begin{gather*}
\Upsilon_{\alpha\alpha}=0, \quad \Upsilon_{01}=\Upsilon_{10}=\frac{\mathrm{d}u} {\sqrt{1+u^2}}, \quad
\Upsilon_{02}= \Upsilon_{20} = \frac{c^2u^2a(h(v))\mathrm{d}v}{ \sqrt{c^2 u^2 a(h(v))^2+h'(v)^2 } }, \\
\Upsilon_{03}= \Upsilon_{30} = -\frac{c u h'(v) \mathrm{d}v}{\sqrt{c^2 u^2 a(h(v))^2+h'(v)^2 }}, \quad
\Upsilon_{12}=-\Upsilon_{21}=- \frac{c^2u\sqrt{1+u^2}a(h(v)) \mathrm{d}v}{ \sqrt{c^2 u^2 a(h(v))^2+h'(v)^2 } }, \\
\Upsilon_{13}= -\Upsilon_{31} = \frac{c \sqrt{1+u^2} h'(v)\mathrm{d} v}{ \sqrt{c^2 u^2 a(h(v))^2+h'(v)^2 }}, \\
\Upsilon_{23}= -\Upsilon_{32} = -c \frac{a(h(v))h'(v)\mathrm{d}u
+u\big( a'(h(v))h'(v)^2-a(h(v))h''(v) \big)\mathrm{d}v
}{ c^2 u^2 a(h(v))^2+h'(v)^2  }
\end{gather*}
A straightforward computation shows $\mathrm{d}\Upsilon+\Upsilon\wedge\Upsilon= 0$. Next, the tangent vector $T$ is 
\[T \equiv \chi_*T= \produ{\partial_t,e_1}e_1+\produ{\partial_t,e_2}e_2 = \frac{h'(v)}{\sqrt{c^2 u^2 a(h(v))^2+h'(v)^2 } } e_2.
\]
Its dual 1-form and the associated functions $T_{\alpha}$ are
\[
\eta = h'(v) \mathrm{d}v, \ T_0= T_1 =0, \ T_2=\frac{h'(v)}{ \sqrt{c^2 u^2 a(h(v))^2+h'(v)^2 } }, \  
 T_3=\frac{c u a(h(v))}{\sqrt{c^2 u^2 a(h(v))^2+h'(v)^2 }}.
 \]
Clearly, it holds $\mathrm{d}\eta = 0$. Also, note that $T_{\alpha} = B_{3\alpha}$. 
 We recall that $\Omega = (\omega_{\alpha\beta} ) = B^{-1}\mathrm{d}B+\X$, where
 $\X=(\X_{\alpha\beta})$, 
$\X_{\alpha\beta}= \ep \frac{a'\circ \pi}{a\circ \pi}\Big(
B_{n+1\beta}\omega_{\alpha}-\ep_{\alpha}\ep_{\beta}B_{n+1\alpha}\omega_{\beta} \Big)$,
\begin{align*}
\X_{\alpha\alpha}=\X_{0\alpha}=\X_{\alpha0}=0, \quad 
\X_{12}=-\X_{21}=\frac{a'(h(v))h'(v)\mathrm{d}u}{\sqrt{1+u^2} \sqrt{c^2 u^2 a(h(v))^2+h'(v)^2 } }, \\
\X_{13}=-\X_{31}=\frac{ c u a(h(v))a'(h(v))\mathrm{d}u  }{\sqrt{1+u^2} \sqrt{c^2 u^2 a(h(v))^2+h'(v)^2 }}, \quad 
\X_{23}=-\X_{32}= cu a'(h(v)) \mathrm{d}v,
\end{align*}
In this way, 
\begin{align*}
&\Omega = (\omega_{\alpha\beta}), \quad 
\omega_{\alpha\alpha}=0, \quad 
\omega_{01}=\omega_{10}=\frac{\mathrm{d}u}{\sqrt{1+u^2 }}, \quad
\omega_{02}= \omega_{20}= \frac{c^2u^3 a(h(v)) \mathrm{d}v }{ \sqrt{c^2 u^2 a(h(v))^2+h'(v)^2 } }, 
\\
&\omega_{03}=\omega_{30} =  \frac{-c u h'(v)\mathrm{d}v}{  \sqrt{c^2 u^2 a(h(v))^2+h'(v)^2 } }, 
\\
&\omega_{12}=-\omega_{21} = \frac{
a'(h(v))h'(v)\mathrm{d}u -c^2(u+u^3)a(h(v))\mathrm{d}v 
}{\sqrt{1+u^2}  \sqrt{c^2 u^2 a(h(v))^2+h'(v)^2 }}, 
\\
&\omega_{13}=-\omega_{31}=c\frac{ u a(h(v))a'(h(v))\mathrm{d}u+(1+u^2)h'(v)\mathrm{d}v }
{ \sqrt{1+u^2}  \sqrt{c^2 u^2 a(h(v))^2+h'(v)^2 }  }, \\
&\omega_{23}=-\omega_{32} = 
\frac{c a(h(v))h'(v)\mathrm{d}u+
c u[ a'(h(v))(c^2u^2a(h(v))^2+2h'(v)^2)-a(h(v))h''(v) ]\mathrm{d}v }
{c^2 u^2 a(h(v))^2+h'(v)^2}.
\end{align*}
Now, we compute the shape operator. Since $\omega_{i3}(X) = -\ep_i\produ{AX,e_i}=-\omega_i(AX)$,
then
\begin{align*}
AX &=\sum_i \omega_i(AX)e_i =-\sum_i\omega_{i3}(X)e_i \\
&=-c\frac{ u a(h(v))a'(h(v))\mathrm{d}u(X)+(1+u^2)h'(v)\mathrm{d}v(X) }
{ \sqrt{1+u^2}  \sqrt{c^2 u^2 a(h(v))^2+h'(v)^2 }  }\,e_1 \\
&- \frac{c a(h(v))h'(v)\mathrm{d}u(X)+
c u[ a'(h(v))(c^2u^2a(h(v))^2+2h'(v)^2)-a(h(v))h''(v) ]\mathrm{d}v(X) }
{c^2 u^2 a(h(v))^2+h'(v)^2}\,e_2,
\\
Ae_1 &= \frac{-a'(h(v))h'(v) }{ a(h(v)) \sqrt{c^2 u^2 a(h(v))^2+h'(v)^2 } } e_2, \,\quad
Ae_2 =  \frac{-c^2 u a(h(v)) \Big(ue_1-\sqrt{1+u^2}e_2\Big) }{ \sqrt{c^2 u^2 a(h(v))^2+h'(v)^2 } }.
\end{align*}

\end{example}

\section{Conclusions}

It is well-known that a non-degenerate hypersurface of a semi-Riemannian manifold must satisfy  Gau\ss\, and Codazzi equations.  Our main concern is the converse problem. Indeed, we show that  a semi-Riemannian manifold endowed with a tensor which plays the rule of a second fundamental form, satisfying the Gau\ss\, and Codazzi equations, and extra condition is needed to obtain a local isometric immersion as a non-degenerate hypersurface of a warped product of an interval and a semi-Riemannian space of constant curvature. Indeed, among all conditions of Definition \ref{structureconditions}, equation (\ref{condC}) cannot be deduced from Codazzi (\ref{eq-codazziM}) and Gau\ss\, (\ref{eq-GaussM}) equations. This means that, in general, one cannot consider a Riemannian manifold endowed with a second fundamental form, and think of it as a spacelike hypersurface of \textit{some} spacetime. However, if one fixes the spacetime first and then consider a hypersurface, everything works as expected. 

\section*{Acknowledgments}

The second author is partially supported by the Spanish MEC-FEDER Grant
MTM2007-60731 and by the Junta de Andaluc\'ia Grant P09-FQM-4496 (with FEDER funds.)

\end{document}